\documentclass{amsart}
\usepackage{verbatim,amsmath,amssymb,amscd,color,graphics}
\usepackage[enableskew]{youngtab}
\usepackage[all]{xy}
\usepackage[dvips]{graphicx}
\usepackage{tikz}
\usetikzlibrary{arrows,matrix}
\tikzset{tab/.style={matrix of math nodes,column sep=-.4, row sep=-.4,text height=9pt,text width=8pt,align=center,inner sep=1.5}}
\tikzset{dynkdot/.style={circle,draw,scale=.62}}
\usepackage[colorlinks=true, pdfstartview=FitV, linkcolor=blue, citecolor=blue, urlcolor=blue]{hyperref}

\usepackage{listings}
\lstdefinelanguage{Sage}[]{Python}
{morekeywords={False,sage,True},sensitive=true}
\definecolor{dblackcolor}{rgb}{0.0,0.0,0.0}
\definecolor{dbluecolor}{rgb}{0.01,0.02,0.7}
\definecolor{dgreencolor}{rgb}{0.2,0.4,0.0}
\definecolor{dgraycolor}{rgb}{0.30,0.3,0.30}

\usepackage{xparse}

\makeatletter
\protected\def\specialmergetwolists{%
  \begingroup
  \@ifstar{\def\cnta{1}\@specialmergetwolists}
    {\def\cnta{0}\@specialmergetwolists}%
}
\def\@specialmergetwolists#1#2#3#4{%
  \def\tempa##1##2{%
    \edef##2{%
      \ifnum\cnta=\@ne\else\expandafter\@firstoftwo\fi
      \unexpanded\expandafter{##1}%
    }%
  }%
  \tempa{#2}\tempb\tempa{#3}\tempa
  \def\cnta{0}\def#4{}%
  \foreach \x in \tempb{%
    \xdef\cnta{\the\numexpr\cnta+1}%
    \gdef\cntb{0}%
    \foreach \y in \tempa{%
      \xdef\cntb{\the\numexpr\cntb+1}%
      \ifnum\cntb=\cnta\relax
        \xdef#4{#4\ifx#4\empty\else,\fi\x#1\y}%
        \breakforeach
      \fi
    }%
  }%
  \endgroup
}
\makeatother

\usepackage{xparse}
\DeclareDocumentCommand\rpp{ m m g }{
	\foreach \x [count=\s from 1] in {#1}{
	        {\ifnum\s=1
	                \draw (0,-\s)--(\x,-\s);
	                \fi}
	   \draw (0,-\s-1) to (\x,-\s-1);
	   \foreach \y in {0, ..., \x} {\draw (\y,-\s)--(\y,-\s-1);}
	}
	\specialmergetwolists{/}{#1}{#2}\ziplist
	\foreach \x/\y [count=\yi from 1] in \ziplist{
	    \node[anchor=west] at (\x,-\yi - .5) {$\y$};  % font=\scriptsize
	}
	\IfValueT {#3}
	{\foreach \z [count=\zi from 1] in {#3} {\node[anchor=east] at (0,-\zi - .5) {$\z$};}}  %font=\scriptsize
	{}
}
\numberwithin{equation}{section}
 
\newtheorem{theorem}{Theorem}
\newtheorem{proposition}[theorem]{Proposition}
\newtheorem{lemma}[theorem]{Lemma}
\newtheorem{corollary}[theorem]{Corollary}

\theoremstyle{definition}

\newtheorem{definition}[theorem]{Definition}
\newtheorem{remark}[theorem]{Remark}
\newtheorem{example}[theorem]{Example}
 
\numberwithin{theorem}{section}

\newcommand{\g}{\mathfrak{g}}

\newcommand{\ZZ}{\mathbb{Z}}

\newcommand{\p}{\mathcal{P}}
\newcommand{\rc}{\mathcal{RC}} % rigged configurations
\newcommand{\RC}{\operatorname{RC}} % unrestricted rigged configurations
\newcommand{\virtual}[1]{\widehat{#1}} % virtual correspondence
\newcommand{\vp}{\p}%{\virtual{\p}}
\newcommand{\vrc}{\rc}%{\virtual{\rc}}
\newcommand{\lh}{\operatorname{lh}}
\newcommand{\lhs}{\lh_{\mathrm{sp}}} % lh spin
\newcommand{\lhsv}{\virtual{\lh}_{\mathrm{sp}}} % virtual lh spin
\newcommand{\lb}{\operatorname{lb}}
\newcommand{\ls}{\operatorname{ls}}
\newcommand{\lx}{\operatorname{lx}}
\newcommand{\deltas}{\delta_{\mathrm{sp}}} % \delta spin
\newcommand{\deltasv}{\virtual{\delta}_{\mathrm{sp}}} % virtual \delta spin
\newcommand{\deltast}{\widetilde{\delta}_{\mathrm{sp}}} % \delta spin tilde
\newcommand{\rh}{\operatorname{rh}}
\newcommand{\rb}{\operatorname{rb}}
\newcommand{\rs}{\operatorname{rs}}
\newcommand{\cc}{\operatorname{cc}} % cocharge
\newcommand{\id}{\operatorname{id}} % Identity map
\newcommand{\emb}{\operatorname{emb}}
\newcommand{\wt}{\operatorname{wt}}
\newcommand{\sgn}{\operatorname{sgn}}  % signature
\newcommand{\la}{\lambda}

\newcommand{\ot}{\otimes}

\newcommand{\lusz}{\star} % Lusztig's involution
\newcommand{\hwstar}{\operatorname{\diamond}} % to_hw \circ Lusztig

\newcommand{\xmap}[1]{\xrightarrow[\hspace{30pt}]{#1}}

\newcommand{\iso}{\cong} % \simeq

\newcommand{\mone}{\overline{1}}

\newcommand{\mn}{\overline{n}}

\newcommand{\clfw}{\varpi} % classical fundamental weight

\definecolor{darkred}{rgb}{0.7,0,0} % darkred color
\newcommand{\defn}[1]{{\color{darkred}\emph{#1}}} % emphasis of a definition

\usepackage[colorinlistoftodos]{todonotes}

\begin{document}
 
\title[RC bijection for nonexceptional affine types]
{Rigged configuration bijection and proof of the $X=M$ conjecture for nonexceptional affine types}

\author[M.~Okado]{Masato Okado}
\address[M. Okado]{Department of Mathematics, Osaka City University, 3-3-138,
Sugimoto, Sumiyoshi-ku, Osaka, 558-8585, Japan}
\email{okado@sci.osaka-cu.ac.jp}

\author[A.~Schilling]{Anne Schilling}
\address[A. Schilling]{Department of Mathematics, University of California, One Shields
Avenue, Davis, CA 95616-8633, U.S.A.}
\email{anne@math.ucdavis.edu}
\urladdr{http://www.math.ucdavis.edu/\~{}anne}

\author[T.~Scrimshaw]{Travis Scrimshaw}
\address[T. Scrimshaw]{School of Mathematics and Physics, University of Queensland, St. Lucia, QLD 4072, Australia}
\email{tcscrims@gmail.com}
\urladdr{https://sites.google.com/view/tscrim/home}
 
  
\begin{abstract}
We establish a bijection between rigged configurations and
highest weight elements of a tensor  product of Kirillov--Reshetikhin crystals for all nonexceptional types.
A key idea for the proof is to embed both objects into bigger sets for 
simply-laced types $A_n^{(1)}$ or $D_n^{(1)}$, whose bijections have already been established.
As a consequence we settle the $X=M$ conjecture in full generality for nonexceptional types.
Furthermore, the bijection extends to a classical crystal isomorphism and sends the combinatorial 
$R$-matrix to the identity map on rigged configurations.
\end{abstract}
 
\maketitle
 
\tableofcontents

\pagebreak

%%%%%%%%%%%%%%%%%%%%%%%%%%%%%%%%%%%%%%%%%%%%%%%%%%%
\section{Introduction}
%%%%%%%%%%%%%%%%%%%%%%%%%%%%%%%%%%%%%%%%%%%%%%%%%%%

Kerov, Kirillov and Reshetikhin~\cite{KKR:1986} introduced rigged configurations as combinatorial objects to 
parameterize the Bethe vectors for Heisenberg spin chains. Moreover, they constructed a bijection from rigged 
configurations to highest weight elements of a tensor product of the vector representation of $\mathfrak{sl}_2$.
This was generalized to tensor products of the symmetric tensor representations of $\mathfrak{sl}_n$ in~\cite{KR86}, 
where the generating function of rigged configurations with the cocharge statistic was shown to be the Kostka polynomials.
This was generalized further to tensor products of multiples of fundamental weights (which can be interpreted as rectangles) 
in~\cite{KSS:2002}, where these rigged configurations were connected with Littlewood--Richardson tableaux and 
generalized Kostka polynomials~\cite{KS02, Shimozono01, Shimozono01b, Shimozono02, SW99, SW00}.

In order to generalize this beyond type $A$, the notion of Kirillov--Reshetihkin (KR) crystals is needed.
Let $\g$ be an affine Kac--Moody Lie algebra and $U'_q(\g)$ be the quantum group of $\g' := [\g, \g]$.
A KR crystal is the crystal basis of a Kirillov--Reshetikhin module:
a certain finite-dimensional $U'_q(\g)$-module that is the minimal affinization of a multiple of a fundamental 
weight~\cite{Chari95, CP95II, CP95, CP96, CP96II, CP98}.
KR crystals were shown to exist for nonexceptional types in~\cite{OS08} and their $U_q'(\g)$-crystal structure 
was given in~\cite{FOS:2009}. A path is a classically highest weight element in the tensor product of KR crystals, 
where for type $A_n^{(1)}$, this agrees with the notion above.

Paths also arise from calculations of 2D integrable lattice models using Baxter's corner transfer matrix method~\cite{B89}.
This method leads to the quantity $X$, which is the sum over the intrinsic energy statistic of paths of $B$, a tensor product of KR crystals.
Let $M$ denote the sum over the cocharge statistic of rigged configurations of a fixed multiplicity matrix $L$, which is called the fermionic formula.
The $X = M$ conjecture of~\cite{HKOTY99, HKOTT02} states that $X$ is equal to $M$ when $L$ counts the factors in $B$,
which suggests the existence of a bijection between paths and rigged configurations that sends the intrinsic 
energy to cocharge.

As previously mentioned, the general case of the desired bijection was proven to be a bijection in type 
$A_n^{(1)}$~\cite{KSS:2002}, building upon~\cite{KKR:1986, KR86}.
Furthermore, for type $A_n^{(1)}$, the bijection was extended as a classical crystal isomorphism in~\cite{DS06} 
using the crystal structure of~\cite{Sch:2006} and to a full $U_q'(\g)$-crystal isomorphism in~\cite{SW10}.
For type $D_n^{(1)}$, an analogous bijection was proven in the general case in~\cite{OSSS16}, building upon the 
special cases of~\cite{OSS:2013, S:2005}; the bijection was shown to intertwine with the
classical crystal structure in~\cite{Sak:2013}.
For other types, the bijection is known in a number of special 
cases~\cite{OSS:2003a, OSS03III, OSS03II, OS12, OSS:2013, SS:X=M, schilling.scrimshaw.2015, scrimshaw.2015, 
Scrimshaw17}.

As far as the $X=M$ conjecture is concerned, there is an alternative proof by 
Naoi~\cite{naoi.2012} for type $A_n^{(1)}$ and $D_n^{(1)}$ using representation theory of 
Kirillov--Reshetikhin modules for a current algebra.

KR crystals are also known to be a reservoir of perfect crystals~\cite{KKMMNN91, KKMMNN92}.
Indeed, the condition for a KR crystal to be perfect of a fixed level $\ell$ is proven for all nonexceptional types 
in~\cite{FOS10}. In particular, this allows KR crystals to be used in the Kyoto path model of~\cite{KKMMNN92}, 
an iterative method to construct highest weight $U_q(\g)$-crystals from the KR crystal.
If all components of the tensor product of KR crystals $B$ are perfect of level $\ell$, 
then it is known that $X$, in a suitable limit, turns our to be a branching function for the coset $\g/\g_0$, where $\g_0$ is
its underlying finite-dimensional simple Lie algebra of $\g$.
Hence, the proof of $X=M$ conjecture implies a fermionic formula of such branching functions
(see~\cite[Theorem 5.4]{HKOTY99} and~\cite[Proposition 4.5]{HKOTT02}).

The goal of this paper is to give a combinatorial proof of the $X = M$ conjecture for all nonexceptional types.
Our main result is the construction of an explicit bijection $\Phi$ from paths to rigged configurations for all 
nonexceptional affine types. Furthermore, we show that $\Phi$ sends the combinatorial $R$-matrix on paths to
the identity on rigged configurations and can be extended to a classical crystal isomorphism.
The bijection also sends the intrinsic energy to cocharge up to a simple involution $\theta$ on rigged configurations 
which interchanges each rigging with its corigging. The map $\theta$ is related to the Lusztig involution.
These facts were shown in type $A_n^{(1)}$ (resp.~$D_n^{(1)}$) in~\cite{KSS:2002} (resp.~\cite{OSSS16}).

Our techniques for constructing $\Phi$ use virtual crystals~\cite{K96, OSS03III, OSS03II, schilling.scrimshaw.2015}, 
which are crystals constructed using diagram foldings of a simply-laced type.
Explicitly, we construct $\Phi$ for nonexceptional type $\g$ by lifting to ambient type $A$ or $D$, doing the 
type $A$ or $D$ bijection and then retracting back to type $\g$.
Then, using the algorithm for the bijection in simply-laced types, we give an explicit algorithm for $\Phi$ in all 
nonexceptional types, where the basic operation $\delta$ is given in~\cite{OSS:2003a}.

In~\cite{FOS:2009}, KR crystals were constructed using Kashiwara--Nakashima (KN) tableaux~\cite{KN:1994}.
However, this does not explicitly describe the tableaux in the image of $\Phi^{-1}$ except in type $A_n^{(1)}$.
By taking the image of the bijection $\Phi^{-1}$ on a single KR crystal, a new tableau model is obtained, coined 
Kirillov--Reshetikhin (KR) tableau.
KR tableaux have been explicitly described on classically highest weight elements for all nonexceptional 
types~\cite{OSS:2013, schilling.scrimshaw.2015}, type $G_2^{(1)}$ and $D_4^{(3)}$~\cite{scrimshaw.2015} and 
some additional special cases~\cite{Scrimshaw17}.
From~\cite{OSSS16}, the image of $\Phi^{-1}$ in type $D_n^{(1)}$ for arbitrary factors is precisely described by a 
tensor product of KR tableaux, which gives an explicit algorithm for computing $\Phi$.
Furthermore, KR tableaux distinguish each classical component.
As a consequence of our construction, we have that the image of $\Phi^{-1}$ are tensor products of KR tableaux 
for all nonexceptional types.

We expect our techniques to apply to the exceptional cases as well.
To achieve this, the existence of KR crystals for exceptional type $\g$ needs to be established first,
which is sufficient for extensions to type $G_2^{(1)}$ and $D_4^{(3)}$,
as well as the bijection for types $E_{6,7,8}^{(1)}$ in full generality to obtain the remaining types, $F_4^{(1)}$ and $E_6^{(2)}$.

In~\cite{Scrimshaw17}, a more conceptual approach was given to describe $\Phi$ as relating to the crystal s
tructure of the factors added.
In~\cite{SalisburyS16II}, the map $\theta$ was shown to be the star involution on the rigged configuration model 
for $B(\infty)$ given in~\cite{SalisburyS15, SalisburyS15II}.
Our results are more evidence that $\Phi$ has a natural crystal-theoretic description since $\Phi$ respects the 
virtualization of the KR crystals.
Furthermore, our results might have applications to construct extremal level-zero crystals~\cite{Kashiwara.2002} 
using rigged configurations, parallel to~\cite{HN06, LL15, NS03, NS05II, NS06, NS08II, NS08, PS15}.
This would allow a direct description of a $U_q(\g)$-crystal structure on rigged configuration.

Soliton cellular automata (SCA) are nonlinear discrete dynamical systems that are generalizations of the 
Takahashi--Satsuma box-ball system~\cite{TS} and have been well-studied, 
{\it e.g.},~\cite{Mohamad, FOY, HHIKTT, HKT00, HKT:2001, HKOTY02, KTT04, LS17, MOW12, MW13, 
TNS96, Yamada04, Yamada07}.
The bijection $\Phi$ can be considered as a linearization of the dynamics, and hence, 
properties of SCA become easy to show by making the most of $\Phi$~\cite{LS17}.
In particular, the rigged configurations encode the action-angle variables of the SCA~\cite{KOSTY:2006, Takagi05}.
Moreover, the bijection $\Phi^{-1}$ in type $A_n^{(1)}$ was shown to be described by a tropicalization of the $\tau$ 
function from the Kadomtsev--Petviashili (KP) hierarchy~\cite{KSY:2006}.
Our results prove the basic assumptions of~\cite{LS17} in all nonexceptional types, which gives an interpretation of~\cite[Prop.~4.10]{HKOTT02}.

This paper is organized as follows.
In Section~\ref{sec:background}, we give the necessary background on crystals and rigged configurations.
In Section~\ref{sec:bijection}, we construct the bijection $\Phi$ and prove our main results.
In Section~\ref{sec:properties}, we show some properties of $\Phi$ and prove the $X = M$ conjecture for nonexceptional types.

%%%%%%%%%%%%%%%%%%%%%%%%%%%%%%%%%%%%%%%%%%%%%%%%%%%%%%%%%%%%%%
\subsection*{Acknowledgments}
MO was partially supported by JSPS grant 16H03922.
AS was partially supported by NSF grant  DMS--1500050.
TS was partially supported by NSF RTG grant DMS--1148634 and JSPS grant K15K13429.

TS and MO would like to thank the University of California Davis for hospitality during their stay in March, 2017, 
where the majority of this work took place.
TS would like to thank Osaka City University for hospitality during his visit in July, 2017.
This work benefitted from computations and experimentations in {\sc Sage}~\cite{combinat,sage}.

The authors would like to thank the anonymous referee for useful comments.

%%%%%%%%%%%%%%%%%%%%%%%%%%%%%%%%%%%%%%%%%%%%%%%%%%%
\section{Background}
\label{sec:background}
%%%%%%%%%%%%%%%%%%%%%%%%%%%%%%%%%%%%%%%%%%%%%%%%%%%

%%%%%%%%%%%%%%%%%%%%%%%%%%%%%%%%%%%%%%%%%%%%%%%%%%%
\subsection{Crystals}

Let $\g$ be an affine Kac--Moody Lie algebra with index set $I$, Cartan matrix $(A_{ij})_{i,j \in I}$, simple roots 
$(\alpha_i)_{i \in I}$, simple coroots $(\alpha_i^{\vee})_{i \in I}$, fundamental weights $(\Lambda_i)_{i \in I}$, weight 
lattice $P$ and canonical pairing $\langle \cdot ,\cdot \rangle$ such that $\langle \alpha_i^{\vee}, \alpha_j \rangle = A_{ij}$. 
We follow the labeling of $I$ given in~\cite{kac}.
For our purpose we also need the opposite labeling of $I$ for $A^{(2)}_{2n}$ denoted by $A^{(2)\dagger}_{2n}$.
The Dynkin diagrams of all nonexceptional types including  $A^{(2)\dagger}_{2n}$ are given in Table \ref{tab:Dynkin}.
Let $U_q(\g)$ denote the corresponding quantum group, and let $U_q'(\g) = U_q([\g, \g])$ be the quantum group 
corresponding to the derived subalgebra of $\g$.
Let $\g_0$ denote the canonical simple Lie algebra given by the index set $I_0 := I \setminus \{0\}$.
Let $\clfw_i$ ($i\in I_0$) denote the funamental weights of the weight lattice $P_0$ of type $\g_0$.

%%%%%%%%%%%%%%%%%%%%%%%%%%%%%%%%%%%%%%%%%%%%%%%%%%%%%
{\unitlength=.95pt
\begin{table}
\begin{tabular}[t]{rlrl}
$A_1^{(1)}$:&
\begin{picture}(26,20)(-5,-5)
\multiput( 0,0)(20,0){2}{\circle{6}}
\multiput(2.85,-1)(0,2){2}{\line(1,0){14.3}}
\put(0,-5){\makebox(0,0)[t]{$0$}}
\put(20,-5){\makebox(0,0)[t]{$1$}} \put( 6, 0){\makebox(0,0){$<$}}
\put(14, 0){\makebox(0,0){$>$}}
\end{picture}
&
$A^{(2)}_2$:&
\begin{picture}(26,20)(-5,-5)
\multiput( 0,0)(20,0){2}{\circle{6}}
\multiput(2.958,-0.5)(0,1){2}{\line(1,0){14.084}}
\multiput(2.598,-1.5)(0,3){2}{\line(1,0){14.804}}
\put(0,-5){\makebox(0,0)[t]{$0$}}
\put(20,-5){\makebox(0,0)[t]{$1$}} \put(10,0){\makebox(0,0){$<$}}
%\put(0,13){\makebox(0,0)[t]{$2$}}
%\put(20,13){\makebox(0,0)[t]{$2$}}
\end{picture}
\\
&
\\
\begin{minipage}{4em}
\begin{flushright}
$A_n^{(1)}$:\\$(n \geqslant 2)$
\end{flushright}
\end{minipage}&
\begin{picture}(106,40)(-5,-5)
\multiput( 0,0)(20,0){2}{\circle{6}}
\multiput(80,0)(20,0){2}{\circle{6}} \put(50,20){\circle{6}}
\multiput( 3,0)(20,0){2}{\line(1,0){14}}
\multiput(63,0)(20,0){2}{\line(1,0){14}}
\multiput(39,0)(4,0){6}{\line(1,0){2}}
\put(2.78543,1.1142){\line(5,2){44.429}}
\put(52.78543,18.8858){\line(5,-2){44.429}}
\put(0,-5){\makebox(0,0)[t]{$1$}}
\put(20,-5){\makebox(0,0)[t]{$2$}}
\put(80,-5){\makebox(0,0)[t]{$n\!\! -\!\! 1$}}
\put(100,-5){\makebox(0,0)[t]{$n$}}
\put(55,20){\makebox(0,0)[lb]{$0$}}
\end{picture}
&
\begin{minipage}[b]{4em}
\begin{flushright}
$A_{2n}^{(2)}$:\\$(n \geqslant 2)$
\end{flushright}
\end{minipage}&
\begin{picture}(126,20)(-5,-5)
\multiput( 0,0)(20,0){3}{\circle{6}}
\multiput(100,0)(20,0){2}{\circle{6}}
\multiput(23,0)(20,0){2}{\line(1,0){14}}
\put(83,0){\line(1,0){14}}
\multiput( 2.85,-1)(0,2){2}{\line(1,0){14.3}} %double line
\multiput(102.85,-1)(0,2){2}{\line(1,0){14.3}} %double line
\multiput(59,0)(4,0){6}{\line(1,0){2}} %dash line
\put(10,0){\makebox(0,0){$<$}} \put(110,0){\makebox(0,0){$<$}}
\put(0,-5){\makebox(0,0)[t]{$0$}}
\put(20,-5){\makebox(0,0)[t]{$1$}}
\put(40,-5){\makebox(0,0)[t]{$2$}}
\put(100,-5){\makebox(0,0)[t]{$n\!\! -\!\! 1$}}
\put(120,-5){\makebox(0,0)[t]{$n$}}
%\put(0,13){\makebox(0,0)[t]{$2$}}
%\put(20,13){\makebox(0,0)[t]{$2$}}
%\put(40,13){\makebox(0,0)[t]{$2$}}
%\put(100,13){\makebox(0,0)[t]{$2$}}
%\put(120,13){\makebox(0,0)[t]{$2$}}
%\put(120,13){\makebox(0,0)[t]{$2$}}
\end{picture}
\\
&
\\
\begin{minipage}[b]{4em}
\begin{flushright}
$B_n^{(1)}$:\\$(n \geqslant 3)$
\end{flushright}
\end{minipage}&
\begin{picture}(126,40)(-5,-5)
\multiput( 0,0)(20,0){3}{\circle{6}}
\multiput(100,0)(20,0){2}{\circle{6}} \put(20,20){\circle{6}}
\multiput( 3,0)(20,0){3}{\line(1,0){14}}
\multiput(83,0)(20,0){1}{\line(1,0){14}}
\put(20,3){\line(0,1){14}}
\multiput(102.85,-1)(0,2){2}{\line(1,0){14.3}} %double line
\multiput(59,0)(4,0){6}{\line(1,0){2}} %dash line
\put(110,0){\makebox(0,0){$>$}} \put(0,-5){\makebox(0,0)[t]{$1$}}
\put(20,-5){\makebox(0,0)[t]{$2$}}
\put(40,-5){\makebox(0,0)[t]{$3$}}
\put(100,-5){\makebox(0,0)[t]{$n\!\! -\!\! 1$}}
\put(120,-5){\makebox(0,0)[t]{$n$}}
\put(25,20){\makebox(0,0)[l]{$0$}}
%\put(120,13){\makebox(0,0)[t]{$2$}}
\end{picture}
&
$A^{(2)\dagger}_2$:&
\begin{picture}(26,20)(-5,-5)
\multiput( 0,0)(20,0){2}{\circle{6}}
\multiput(2.958,-0.5)(0,1){2}{\line(1,0){14.084}}
\multiput(2.598,-1.5)(0,3){2}{\line(1,0){14.804}}
\put(0,-5){\makebox(0,0)[t]{$0$}}
\put(20,-5){\makebox(0,0)[t]{$1$}} \put(10,0){\makebox(0,0){$>$}}
\put(0,13){\makebox(0,0)[t]{$$}} \put(20,13){\makebox(0,0)[t]{$$}}
\end{picture}
\\
&
\\
\begin{minipage}[b]{4em}
\begin{flushright}
$C_n^{(1)}$:\\$(n \geqslant 2)$
\end{flushright}
\end{minipage}&
\begin{picture}(126,20)(-5,-5)
\multiput( 0,0)(20,0){3}{\circle{6}}
\multiput(100,0)(20,0){2}{\circle{6}}
\multiput(23,0)(20,0){2}{\line(1,0){14}}
\put(83,0){\line(1,0){14}}
\multiput( 2.85,-1)(0,2){2}{\line(1,0){14.3}} %double line
\multiput(102.85,-1)(0,2){2}{\line(1,0){14.3}} %double line
\multiput(59,0)(4,0){6}{\line(1,0){2}} %dash line
\put(10,0){\makebox(0,0){$>$}} \put(110,0){\makebox(0,0){$<$}}
\put(0,-5){\makebox(0,0)[t]{$0$}}
\put(20,-5){\makebox(0,0)[t]{$1$}}
\put(40,-5){\makebox(0,0)[t]{$2$}}
\put(100,-5){\makebox(0,0)[t]{$n\!\! -\!\! 1$}}
\put(120,-5){\makebox(0,0)[t]{$n$}}

%\put(20,13){\makebox(0,0)[t]{$2$}}
%\put(40,13){\makebox(0,0)[t]{$2$}}
%\put(100,13){\makebox(0,0)[t]{$2$}}
\end{picture}
&
\begin{minipage}[b]{4em}
\begin{flushright}
$A_{2n}^{(2)\dagger}$:\\$(n \geqslant 2)$
\end{flushright}
\end{minipage}&
\begin{picture}(126,20)(-5,-5)
\multiput( 0,0)(20,0){3}{\circle{6}}
\multiput(100,0)(20,0){2}{\circle{6}}
\multiput(23,0)(20,0){2}{\line(1,0){14}}
\put(83,0){\line(1,0){14}}
\multiput( 2.85,-1)(0,2){2}{\line(1,0){14.3}} %double line
\multiput(102.85,-1)(0,2){2}{\line(1,0){14.3}} %double line
\multiput(59,0)(4,0){6}{\line(1,0){2}} %dash line
\put(10,0){\makebox(0,0){$>$}} \put(110,0){\makebox(0,0){$>$}}
\put(0,-5){\makebox(0,0)[t]{$0$}}
\put(20,-5){\makebox(0,0)[t]{$1$}}
\put(40,-5){\makebox(0,0)[t]{$2$}}
\put(100,-5){\makebox(0,0)[t]{$n\!\! -\!\! 1$}}
\put(120,-5){\makebox(0,0)[t]{$n$}}
\put(0,13){\makebox(0,0)[t]{$$}}
\put(20,13){\makebox(0,0)[t]{$$}}
\put(40,13){\makebox(0,0)[t]{$$}}
\put(100,13){\makebox(0,0)[t]{$$}}
\put(120,13){\makebox(0,0)[t]{$$}}
\put(120,13){\makebox(0,0)[t]{$$}}
\end{picture}
\\
&
\\
\begin{minipage}[b]{4em}
\begin{flushright}
$D_n^{(1)}$:\\$(n \geqslant 4)$
\end{flushright}
\end{minipage}&
\begin{picture}(106,40)(-5,-5)
\multiput( 0,0)(20,0){2}{\circle{6}}
\multiput(80,0)(20,0){2}{\circle{6}}
\multiput(20,20)(60,0){2}{\circle{6}} \multiput(
3,0)(20,0){2}{\line(1,0){14}}
\multiput(63,0)(20,0){2}{\line(1,0){14}}
\multiput(39,0)(4,0){6}{\line(1,0){2}}
\multiput(20,3)(60,0){2}{\line(0,1){14}}
\put(0,-5){\makebox(0,0)[t]{$1$}}
\put(20,-5){\makebox(0,0)[t]{$2$}}
\put(80,-5){\makebox(0,0)[t]{$n\!\! - \!\! 2$}}
\put(103,-5){\makebox(0,0)[t]{$n\!\! -\!\! 1$}}
\put(25,20){\makebox(0,0)[l]{$0$}}
\put(85,20){\makebox(0,0)[l]{$n$}}
\end{picture}
&
\begin{minipage}[b]{4em}
\begin{flushright}
$A_{2n-1}^{(2)}$:\\$(n \geqslant 3)$
\end{flushright}
\end{minipage}&
\begin{picture}(126,40)(-5,-5)
\multiput( 0,0)(20,0){3}{\circle{6}}
\multiput(100,0)(20,0){2}{\circle{6}} \put(20,20){\circle{6}}
\multiput( 3,0)(20,0){3}{\line(1,0){14}}
\multiput(83,0)(20,0){1}{\line(1,0){14}}
\put(20,3){\line(0,1){14}}
\multiput(102.85,-1)(0,2){2}{\line(1,0){14.3}} %double line
\multiput(59,0)(4,0){6}{\line(1,0){2}} %dash line
\put(110,0){\makebox(0,0){$<$}} \put(0,-5){\makebox(0,0)[t]{$1$}}
\put(20,-5){\makebox(0,0)[t]{$2$}}
\put(40,-5){\makebox(0,0)[t]{$3$}}
\put(100,-5){\makebox(0,0)[t]{$n\!\! -\!\! 1$}}
\put(120,-5){\makebox(0,0)[t]{$n$}}
\put(25,20){\makebox(0,0)[l]{$0$}}

%\put(120,13){\makebox(0,0)[t]{$2$}}
\end{picture}
\\
&
\\
&&
\begin{minipage}[b]{4em}
\begin{flushright}
$D_{n+1}^{(2)}$:\\$(n \geqslant 2)$
\end{flushright}
\end{minipage}&
\begin{picture}(126,20)(-5,-5)
\multiput( 0,0)(20,0){3}{\circle{6}}
\multiput(100,0)(20,0){2}{\circle{6}}
\multiput(23,0)(20,0){2}{\line(1,0){14}}
\put(83,0){\line(1,0){14}}
\multiput( 2.85,-1)(0,2){2}{\line(1,0){14.3}} %double line
\multiput(102.85,-1)(0,2){2}{\line(1,0){14.3}} %double line
\multiput(59,0)(4,0){6}{\line(1,0){2}} %dash line
\put(10,0){\makebox(0,0){$<$}} \put(110,0){\makebox(0,0){$>$}}
\put(0,-5){\makebox(0,0)[t]{$0$}}
\put(20,-5){\makebox(0,0)[t]{$1$}}
\put(40,-5){\makebox(0,0)[t]{$2$}}
\put(100,-5){\makebox(0,0)[t]{$n\!\! -\!\! 1$}}
\put(120,-5){\makebox(0,0)[t]{$n$}}

%\put(20,13){\makebox(0,0)[t]{$2$}}
%\put(40,13){\makebox(0,0)[t]{$2$}}
%\put(100,13){\makebox(0,0)[t]{$2$}}
\end{picture}
\\ \vspace{5pt}
\end{tabular}
\caption{Dynkin diagrams for all nonexceptional affine types. The labeling of the
nodes by elements of $I$ is specified under or to the right of the nodes. 
} 
\label{tab:Dynkin}
\end{table}}
%%%%%%%%%%%%%%%%%%%%%%%%%%%%%%%%%%%%%%%%%%%%%%%%%%%%%

Let $c_i$ and $c_i^{\vee}$ denote the Kac and dual Kac labels~\cite[Table~Aff1-3]{kac}.
The null root is given by $\delta = \sum_{i \in I} c_i \alpha_i$.
The canonical central element is given by $c = \sum_{i \in I} c_i^{\vee} \alpha_i^{\vee}$.
The normalized (symmetric) invariant form $( \cdot | \cdot ) \colon P \times P \to \ZZ$ is defined by  
$(\alpha_i | \alpha_j) = \frac{c_i^{\vee}}{c_i} A_{ij}$.

A \defn{$U_q(\g)$-crystal} is a nonempty set $B$ together with \defn{crystal operators} $e_i, f_i \colon B \to B \sqcup \{0\}$, 
for $i \in I$, and \defn{weight function} $\wt \colon B \to P$. Let $\varepsilon_i, \varphi_i \colon  B \to \ZZ_{\geqslant 0}$ be 
statistics given by
\[
	\varepsilon_i(b) := \max \{ k \mid e_i^k b \neq 0 \},
	\qquad \qquad
	\varphi_i(b) := \max \{ k \mid f_i^k b \neq 0 \}.
\]
The following conditions should be satisfied:
\begin{itemize}
\item[(1)] $\varphi_i(b) = \varepsilon_i(b) + \langle \alpha^{\vee}_i, \wt(b) \rangle$ for all $b \in B$ and $i \in I$.
\item[(2)] $f_i b = b'$ if and only if $b = e_i b'$ for $b, b' \in B$ and $i \in I$.
\end{itemize}
We say an element $b \in B$ is \defn{$J$-highest weight} if $e_i b = 0$ for all $i \in J \subset I$.

We define the \defn{tensor product} of abstract $U_q(\g)$-crystals $B_1$ and $B_2$ as follows. As a set
the crystal $B_2 \otimes B_1$ is the Cartesian product $B_2 \times B_1$. The crystal operators are defined as:
\begin{align*}
e_i(b_2 \otimes b_1) & := \begin{cases}
e_i b_2 \otimes b_1 & \text{if } \varepsilon_i(b_2) > \varphi_i(b_1), \\
b_2 \otimes e_i b_1 & \text{if } \varepsilon_i(b_2) \leqslant \varphi_i(b_1)\,,
\end{cases}
\\ f_i(b_2 \otimes b_1) & := \begin{cases}
f_i b_2 \otimes b_1 & \text{if } \varepsilon_i(b_2) \geqslant \varphi_i(b_1), \\
b_2 \otimes f_i b_1 & \text{if } \varepsilon_i(b_2) < \varphi_i(b_1)\,,
\end{cases}
\\ \varepsilon_i(b_2 \otimes b_1) & := \max(\varepsilon_i(b_1), \varepsilon_i(b_2) - \langle \alpha_i^{\vee}, \wt(b_1)\rangle )\,,
\\ \varphi_i(b_2 \otimes b_1) & := \max(\varphi_i(b_2), \varphi_i(b_1) + \langle \alpha_i^{\vee} \wt(b_2) \rangle )\,,
\\ \wt(b_2 \otimes b_1) & := \wt(b_2) + \wt(b_1)\,.
\end{align*}

\begin{remark}
In this paper we use the convention for tensor products of crystals as in~\cite{Bump.Schilling.2017}, which is opposite
to the convention used by Kashiwara~\cite{Kashiwara:1991}.
\end{remark}

For abstract $U_q(\g)$-crystals $B_1, \dotsc, B_L$, the action of the crystal operators on the tensor product 
$B := B_L \otimes \cdots \otimes B_2 \otimes B_1$ can be computed by the \defn{signature rule}.
Let $b := b_L \otimes \cdots \otimes b_2 \otimes b_1 \in B$, and for $i \in I$, we write
\[
\underbrace{-\cdots-}_{\varphi_i(b_L)}\
\underbrace{+\cdots+}_{\varepsilon_i(b_L)}\
\cdots\
\underbrace{-\cdots-}_{\varphi_i(b_1)}\
\underbrace{+\cdots+}_{\varepsilon_i(b_1)}\,.
\]
Then by successively deleting consecutive $+-$-pairs (in that order) in the above sequence, we obtain a sequence
\[
\sgn_i(b) :=
\underbrace{-\cdots-}_{\varphi_i(b)}\
\underbrace{+\cdots+}_{\varepsilon_i(b)}\,,
\]
called the \defn{reduced signature}.
Suppose $1 \leqslant j_-\leqslant j_+ \leqslant L$ are such that $b_{j_-}$ contributes the rightmost $-$ in $\sgn_i(b)$ 
and $b_{j_+}$ contributes the leftmost $+$ in $\sgn_i(b)$. Then, we have
\begin{align*}
e_i b &:= b_L \otimes \cdots \otimes b_{j_++1} \otimes e_ib_{j_+} \otimes b_{j_+-1} \otimes \cdots \otimes b_1\,, \\
f_i b &:= b_L \otimes \cdots \otimes b_{j_-+1} \otimes f_ib_{j_-} \otimes b_{j_--1} \otimes \cdots \otimes b_1\,.
\end{align*}

Let $B_1$ and $B_2$ be two $U_q(\g)$-crystals.
A \defn{crystal morphism} $\psi \colon B_1 \to B_2$ is a map $B_1 \sqcup \{0\} \to B_2 \sqcup \{0\}$ with $\psi(0) = 0$ 
such that the following properties hold for all $b \in B_1$:
\begin{itemize}
\item[(1)] If $\psi(b) \in B_2$, then $\wt\bigl(\psi(b)\bigr) = \wt(b)$, $\varepsilon_i\bigl(\psi(b)\bigr) = \varepsilon_i(b)$, 
and $\varphi_i\bigl(\psi(b)\bigr) = \varphi_i(b)$.
\item[(2)] We have $\psi(e_i b) = e_i \psi(b)$ if $\psi(e_i b) \neq 0$ and $e_i \psi(b) \neq 0$.
\item[(3)] We have $\psi(f_i b) = f_i \psi(b)$ if $\psi(f_i b) \neq 0$ and $f_i \psi(b) \neq 0$.
\end{itemize}
An \defn{embedding} (resp.~\defn{isomorphism}) is a crystal morphism such that the induced map 
$B_1 \sqcup \{0\} \to B_2 \sqcup \{0\}$ is an embedding (resp.~bijection). A crystal morphism is \defn{strict} 
if it commutes with all crystal operators.

For further details regarding crystals, see~\cite{Bump.Schilling.2017,HK02}.

%%%%%%%%%%%%%%%%%%%%%%%%%%%%%%%%%%%%%%%%%%%%%%%%%%%
\subsection{Kirillov--Reshetikhin crystals}

Let $\g$ be of nonexceptional affine type. A \defn{Kirillov--Reshetikhin (KR) crystal} is a $U_q'(\g)$-crystal corresponding to a 
Kirillov--Reshetikhin (KR) module~\cite{HKOTY99, HKOTT02, OS08}.
KR crystals are finite crystals since KR modules are finite-dimensional.
Specifically, the KR crystal $B^{r,s}$ (where $r \in I_0$ and $s \in \ZZ_{>0}$) have a \emph{multiplicity free} 
decomposition as $U_q(\g_0)$-crystals:
\[
B^{r,s} \iso \bigoplus_{\lambda} B(\lambda)
\qquad\qquad (\text{as $U_q(\g_0)$-crystals})
\]
for certain (distinct) $\lambda \in P_0^+$.
Here $B(\lambda)$ is the highest weight $U_q(\g_0)$-crystal of highest weight $\lambda \in P_0^+$.
An explicit combinatorial construction of $B^{r,s}$ for all nonexceptional types was given in~\cite{FOS:2009} except for 
type $A_{2n}^{(2)\dagger}$.
For type $A_{2n}^{(2)\dagger}$, we can construct $B^{r,s}$ from the corresponding KR crystal in type $A_{2n}^{(2)}$ 
by relabeling the nodes $i \leftrightarrow n - i$, but we need to be careful about the weight.
In particular, $\g_0$ is type $B_n$ for type $A_{2n}^{(2)\dagger}$, and so we need \emph{twice} $\clfw_n$ of type $B_n$.
Hence $\kappa_r = 1$ unless $\g = A_{2n}^{(2)\dagger}$ and $r = n$, in which case $\kappa_n = 2$.

We note that there is a unique element $u_{s\kappa_r\clfw_r} \in B(s\kappa_r\clfw_r) \subset B^{r,s}$ of weight 
$s \kappa_r \clfw_r$, called the \defn{maximal element}.
Furthermore, it is known that tensor products of KR crystals $\bigotimes_{i=1}^N B^{r_i, s_i}$ are 
connected~\cite{FSS.2007, Okado.2013} with a unique maximal element 
$u_{s_1 \kappa_{r_1} \clfw_{r_1}} \otimes \cdots \otimes u_{s_N \kappa_{r_1} \clfw_{r_N}}$.
Therefore, there exists a unique $U_q'(\g)$-crystal isomorphism $R \colon B \otimes B' \to B' \otimes B$ called the 
\defn{combinatorial $R$-matrix} defined by $R(u \otimes u') = u' \otimes u$, where $u$ and $u'$ are the maximal 
elements of $B$ and $B'$ respectively.

%%%%%%%%%%%%%%%%%%%%%%%%%%%%%%%%%%%%%%%%%%%%%%%%%%%
\subsection{Dualities}
Denote by $\tau \colon I_0 \to I_0$ the (classical) diagram automorphism given by 
$-w_0 \alpha_i = \alpha_{\tau(i)}$ (equivalently 
$-w_0 \clfw_i = \clfw_{\tau(i)}$), where $w_0$ is the longest element of the Weyl group of $\g_0$.
Explicitly, we have
\begin{align*}
\tau(i) &= n+1-i && (\g_0 = A_n),
\\ \tau(i) & = \begin{cases}
i & i \neq n-1,n, \\
n & i = n - 1, \\
n-1 & i = n,
\end{cases}
&& (\g_0 = D_n \text{ for $n$ odd}),
\\ \tau(i) & = i && (\g_0 = D_n \text{ for $n$ even}, B_n, C_n).
\end{align*}
Define the \defn{Lusztig involution} $\lusz \colon B(\lambda) \to B(\lambda)$ as the unique involution satisfying
\begin{equation}
\label{eq:lusztig_involution}
(e_i b)^{\lusz} = f_{\tau(i)} b^{\lusz},
\qquad\qquad
(f_i b)^{\lusz} = e_{\tau(i)} b^{\lusz},
\qquad\qquad
\wt(b^{\lusz}) = w_0 \wt(b).
\end{equation}
We note that the Lusztig involution sends highest weight elements to a lowest weight element.
It can be extended to $\lusz \colon B^{r,s} \to B^{r,s}$ by defining $\tau(0) = 0$ and requiring
that $\lusz$ satisfies~\eqref{eq:lusztig_involution}.

Let $B^{\vee}$ denote the \defn{contragredient dual} crystal of $B$.
As a set $B^\vee=\{b^{\vee} \mid b \in B\}$ with the crystal structure given by
\[
(e_i b)^{\vee} = f_i b^\vee,
\qquad\qquad
(f_i b)^{\vee} = e_i b^\vee,
\qquad\qquad
\wt(b^{\vee}) = -\wt(b).
\]
For the highest weight $U_q(\g_0)$-crystal $B(\lambda)$, we note that $B(\lambda)^{\vee}$ is naturally 
isomorphic to $B(-w_0 \lambda)$.

We can also extend the Lusztig involution and the contragredient dual to tensor products by a natural isomorphism
\begin{equation}
\label{eq:box_factors}
(B_2 \otimes B_1)^{\Box} \iso B_1^{\Box} \otimes B_2^{\Box}
\end{equation}
given by $(b_2 \otimes b_1)^{\Box} = b_1^{\Box} \otimes b_2^{\Box}$, where $\Box \in \{\vee, \lusz\}$.

Next, for $\g$ of type $A_n^{(1)}$ or $D_n^{(1)}$, we consider the diagram automorphism $\sigma$ given by
\begin{align*}
\sigma(i) &= n+1-i \pmod{n+1} && (\g = A_n^{(1)}),
\\ \sigma(i) & = \begin{cases}
i & i \neq n-1,n, \\
n & i = n - 1, \\
n-1 & i = n,
\end{cases}
&& (\g = D_n^{(1)}).
\end{align*}
This induces a twisted crystal isomorphism from $B^{r,s}$ to $B^{\sigma(r),s}$ given by
\[
(f_i b)^{\sigma} = f_{\sigma(i)} b^{\sigma},
\qquad\qquad
(e_i b)^{\sigma} = e_{\sigma(i)} b^{\sigma},
\qquad\qquad
\Lambda_i \mapsto \Lambda_{\sigma(i)}.
\]
By abuse of notation, we denote this twisted crystal isomorphism by $\sigma$.
We can extend $\sigma$ to tensor products by a natural isomorphism
\[
(B_1 \otimes B_2)^{\sigma} \iso B_1^{\sigma} \otimes B_2^{\sigma}.
\]
For type $A_n^{(1)}$, we have $\sigma(b) = b^{\vee\lusz}$~\cite{SS:X=M}.
In type $D_n^{(1)}$ for nonspin columns, we have $\sigma(b)$ by interchanging the letters $n \leftrightarrow \mn$, which 
follows from considering the map on the highest weight $U_q(\g_0)$-crystal $B(\Lambda_1)$.
For the spin columns in type $D_n^{(1)}$, we have $\sigma \colon B(\clfw_{n-1}) \leftrightarrow B(\clfw_n)$ 
with $\sigma(b) = (s_1, \dotsc, s_{n-1}, -s_n)$ for $b = (s_1, \dotsc, s_{n-1}, s_n)$ in the $\pm$-vector description 
of~\cite{KN:1994}.

%%%%%%%%%%%%%%%%%%%%%%%%%%%%%%%%%%%%%%%%%%%%%%%%%%%
\subsection{Virtual crystals}
\label{sec:virtual_crystals}

Virtual crystals were introduced in~\cite{OSS03III, OSS03II} as a way to realize crystals for nonsimply-laced 
types as embeddings into simply-laced types. At the time these papers were written, the existence of some
of the KR crystals had not yet been established, explaining the choice of name ``virtual crystals.''
In the meantime, the existence of all KR crystals of nonexceptional types was established 
in~\cite{O07,OS08} and explicit combinatorial realizations were constructed in~\cite{FOS:2009}. Even though 
this means that virtual crystals are now true realizations of crystals, we will stick with the terminology virtual
crystals for historical reasons.

\begin{figure}
\begin{tikzpicture}[xscale=1.75,yscale=1.25]
\node at (0,0) {$A_{2n-1}^{(1)}$};
\node[dynkdot,label={left:$0$}] (A0) at (1,0) {};
\node[dynkdot,label={above:$1$}] (A1) at (2,-.5){};
\node[dynkdot,label={above:$2$}] (A2) at (3,-.5) {};
\node[dynkdot,label={above:$n-1$}] (A3) at (4,-.5) {};
\node[dynkdot,label={right:$n$}] (A4) at (5,0) {};
\node[dynkdot,label={above:$n+1$}] (A5) at (4,.5){};
\node[dynkdot,label={above:$2n-2$}] (A6) at (3,.5) {};
\node[dynkdot,label={above:$2n-1$}] (A7) at (2,.5) {};
\draw[-] (A0) -- (A1);
\draw[-] (A1) -- (A2);
\draw[dashed] (A2) -- (A3);
\draw[-] (A3) -- (A4);
\draw[-] (A4) -- (A5);
\draw[dashed] (A5) -- (A6);
\draw[-] (A6) -- (A7);
\draw[-] (A7) -- (A0);

\def\Coffset{-1.5}
\node at (-.25,\Coffset+0.22) {$C_n^{(1)}, D_{n+1}^{(2)},$};
\node at (-.25,\Coffset-0.22) {$A_{2n}^{(2)}, A_{2n}^{(2)\dagger}$};
\node[dynkdot,label={below:$0$}] (C0) at (1,\Coffset) {};
\node[dynkdot,label={below:$1$}] (C1) at (2,\Coffset){};
\node[dynkdot,label={below:$2$}] (C2) at (3,\Coffset) {};
\node[dynkdot,label={below:$n-1$}] (C3) at (4,\Coffset) {};
\node[dynkdot,label={below:$n$}] (C4) at (5,\Coffset) {};
\draw[-] (C0.30) -- (C1.150);
\draw[-] (C0.330) -- (C1.210);
\draw[-] (C1) -- (C2);
\draw[dashed] (C2) -- (C3);
\draw[-] (C3.30) -- (C4.150);
\draw[-] (C3.330) -- (C4.210);
\path[-latex,dashed,color=blue,thick]
 (A0) edge (C0)
 (A4) edge (C4);
\foreach \x[evaluate={\dx=8-\x}] in {1,2,3} {
\draw[-latex,dashed,color=blue,thick]
 (A\x) .. controls (\x+.8,-.8) and (\x+.8,-1.2) .. (C\x);
\draw[-latex,dashed,color=blue,thick]
 (A\dx) .. controls (\x+1.5,-.4) and (\x+1.5,-1) .. (C\x);
}
\end{tikzpicture}
\caption{The diagram folding of type $\virtual{\g} = A_{2n-1}^{(1)}$ onto type $\g = C_n^{(1)}, D_{n+1}^{(2)}, A_{2n}^{(2)}, A_{2n}^{(2)\dagger}$.}
\label{fig:diagram_foldings_A}
\end{figure}

\begin{figure}
\begin{tikzpicture}[xscale=1.75,yscale=1.25]
\node at (0,0) {$D_{n+1}^{(1)}$};
\node[dynkdot,label={left:$0$}] (D0) at (1,.5) {};
\node[dynkdot,label={left:$1$}] (D1) at (1,-.5){};
\node[dynkdot,label={above:$2$}] (D2) at (2,0) {};
\node[dynkdot,label={above:$3$}] (D3) at (3,0) {};
\node[dynkdot,label={above:$n-1$}] (D4) at (4,0) {};
\node[dynkdot,label={right:$n$}] (D5) at (5,-.5) {};
\node[dynkdot,label={right:$n+1$}] (D6) at (5,.5) {};
\draw[-] (D0) -- (D2);
\draw[-] (D1) -- (D2);
\draw[-] (D2) -- (D3);
\draw[dashed] (D3) -- (D4);
\draw[-] (D4) -- (D5);
\draw[-] (D4) -- (D6);

\def\Coffset{-1.8}
\node at (0,\Coffset) {$B_n^{(1)}, A_{2n-1}^{(2)}$};
\node[dynkdot,label={left:$0$}] (B0) at (1,\Coffset+.5) {};
\node[dynkdot,label={left:$1$}] (B1) at (1,\Coffset-.5){};
\node[dynkdot,label={below:$2$}] (B2) at (2,\Coffset) {};
\node[dynkdot,label={below:$3$}] (B3) at (3,\Coffset) {};
\node[dynkdot,label={below:$n-1$}] (B4) at (4,\Coffset) {};
\node[dynkdot,label={below:$n$}] (B5) at (5,\Coffset) {};
\draw[-] (B0) -- (B2);
\draw[-] (B1) -- (B2);
\draw[-] (B2) -- (B3);
\draw[dashed] (B3) -- (B4);
\draw[-] (B4.30) -- (B5.150);
\draw[-] (B4.330) -- (B5.210);

\path[-latex,dashed,color=blue,thick]
 (D2) edge (B2)
 (D3) edge (B3)
 (D4) edge (B4);
\draw[-latex,dashed,color=blue,thick]
 (D0) .. controls (0.55,-.05) and (0.55,-0.65) .. (B0);
\draw[-latex,dashed,color=blue,thick]
 (D1) .. controls (0.55,-1) and (0.55,-1.55) .. (B1);
\draw[-latex,dashed,color=blue,thick]
 (D5) .. controls (4.7,-0.8) and (4.7,-1.3) .. (B5);
\draw[-latex,dashed,color=blue,thick]
 (D6) .. controls (5.75,-.3) and (5.75,-1.3) .. (B5);
\end{tikzpicture}
\caption{The diagram folding of type $\virtual{\g} = D_{n+1}^{(1)}$ onto type $\g = B_n^{(1)}, A_{2n-1}^{(2)}$.}
\label{fig:diagram_foldings_D}
\end{figure}

We consider the Dynkin diagram folding that arises from the natural embeddings 
$\g \lhook\joinrel\longrightarrow \virtual{\g}$ given in~\cite{JM85}:
\begin{equation}
\begin{aligned}
	C_n^{(1)}, A_{2n}^{(2)}, A_{2n}^{(2)\dagger}, D_{n+1}^{(2)} & \lhook\joinrel\longrightarrow A_{2n-1}^{(1)},\\
	 B_n^{(1)}, A_{2n-1}^{(2)} & \lhook\joinrel\longrightarrow D_{n+1}^{(1)}.
\end{aligned}
\end{equation}
Let $I^X$ denote the index set of ambient type $X$ ({\it i.e.}, $\virtual{\g}$ is of type $A_{2n-1}^{(1)}$ or $D_{n+1}^{(1)}$).
We denote the corresponding map on the index sets by $\phi \colon I^X \searrow I$ as in 
Figure~\ref{fig:diagram_foldings_A} (resp.~Figure~\ref{fig:diagram_foldings_D}) for $\virtual{\g}$ of type 
$A_{2n-1}^{(1)}$ (resp.~$D_{n+1}^{(1)}$).
For ease of notation, if $Z$ is an object for type $\g$, we denote the corresponding object for type $\virtual{\g}$ 
(type $A_{2n-1}^{(1)}$ or $D_{n+1}^{(1)}$) by $Z^X$.
For example, the weight lattice $P$ is an object for type $\g$, and we denote the corresponding weight lattice for 
type $\virtual{\g}$ by $P^X$. We define the \defn{scaling factors} $\gamma = (\gamma_a)_{a \in I}$ by
\[
	\gamma_a = \begin{cases}
	(2,2,\dotsc,2,1) & \text{for type $B_n^{(1)}$,}\\
	(2,1,\dotsc,1,2) & \text{for type $C_n^{(1)}$,}\\
	(1,1,\dotsc,1,1) & \text{for type $A_{2n-1}^{(2)}$, $D_{n+1}^{(2)}$,}\\
	(1,1,\dotsc,1,2) & \text{for type $A_{2n}^{(2)}$,}\\
	(2,1,\dotsc,1,1) & \text{for type $A_{2n}^{(2)\dagger}$.}
	\end{cases}
\]
Note that if $\lvert\phi^{-1}(a)\rvert \neq 1$, then $\gamma_a = 1$.

Furthermore, we have a natural embedding $\Psi \colon P \to P^X$ given by
\begin{align*}
	\Lambda_a & \mapsto \gamma_a \sum_{b \in \phi^{-1}(a)} \Lambda^X_b,\\
	\alpha_a & \mapsto \gamma_a \sum_{b \in \phi^{-1}(a)} \alpha^X_b,
\end{align*}
where the map on simple roots is induced from the embedding of the root lattice into the weight lattice.
Note that this implies that $\delta \mapsto c_0 \gamma_0 \delta^X$.

\begin{definition}
\label{definition.virtual}
Let $B^X$ be a $U_q'(\virtual{\g})$-crystal and $\virtual{B} \subset B^X$.
Let $\phi$ and $(\gamma_a)_{a \in I}$ be the folding and the scaling factors given above.
The \defn{virtual crystal operators} (of type $\g$) are defined as
\[
	\virtual{e}_a := \prod_{b \in \phi^{-1}(a)} (e^X_b)^{\;\gamma_a},
	\qquad\qquad
	\virtual{f}_a := \prod_{b \in \phi^{-1}(a)} (f^X_b)^{\;\gamma_a}.
\]
A \defn{virtual crystal} is the quadruple $(\virtual{B}, B^X, \phi, (\gamma_a)_{a \in I})$ such that $\virtual{B}$ has a 
$U_q'(\g)$-crystal structure defined by
\begin{equation}
\label{eq:virtual_crystal}
\def\arraystretch{1.3}
\begin{gathered}
	e_a := \virtual{e}_a, \hspace{40pt} f_a := \virtual{f}_a, \\
	\hspace{80pt}\varepsilon_a := \gamma_a^{-1} \varepsilon^X_b, \hspace{40pt} \varphi_a := \gamma_a^{-1} \varphi^X_b, \hspace{20pt} (b \in \phi^{-1}(a))\\
	\wt := \Psi^{-1} \circ \wt^X.
\end{gathered}
\end{equation}
\end{definition}

Consider a set $\virtual{B} \subset B^X$ with a fixed $\phi$ and $(\gamma_a)_{a \in I}$.
If for all $v \in \virtual{B}$ and $a \in I$, we have
\begin{itemize}
\item $\varepsilon^X_b(v) = \varepsilon^X_{b'}(v)$ for all $b, b' \in \phi^{-1}(a)$ and
\item $\varepsilon^X_b(v) / \gamma_a \in \ZZ$ for all $b \in \phi^{-1}(a)$,
\end{itemize}
then we say that $B^X$ is \defn{aligned}.
Note that in particular Definition~\ref{definition.virtual} requires that virtual crystals are aligned.

When there is no danger of confusion, we simply denote the virtual crystal by $\virtual{B}$.
We say that a type $\g$ crystals $B$ is \defn{realized} as a virtual crystal $\virtual{B}$ if there exists a $U_q'(\g)$-crystal 
isomorphism $\chi \colon B \to \virtual{B}$.
We denote the composition of $\chi$ with the natural inclusion $\virtual{B} \subset B^X$ by $\emb \colon B \to B^X$, 
which we call the \defn{virtualization map}.
Furthermore, we will also denote $\virtual{\lambda} = \Psi(\lambda)$.

It is straightforward to see that virtual crystals are closed under direct sums.
Moreover, they are closed under tensor products.

\begin{proposition}[{\cite[Prop.~6.4]{OSS03III}}]
\label{prop:virtual_tensor_category}
Virtual crystals form a tensor category.
\end{proposition}

Next, we provide explicit virtual crystal realizations for all nonexceptional types. We distinguish the cases
that embed into type $A_{2n-1}^{(1)}$ (denoted \defn{ambient type $A$}) and those that embed into
type $D_{n+1}^{(1)}$ (denoted \defn{ambient type $D$}).

We are using the embeddings of types $C_n^{(1)}$, $A_{2n}^{(2)}$, $A_{2n}^{(2)\dagger}$, $D_{n+1}^{(2)}$
into type $A_{2n-1}^{(1)}$ as in~\cite{OSS03III}:
\begin{subequations}
\label{eq:ambient_emb}
\begin{equation}
\label{equation.embedding A}
\virtual{B}^{r,s} = \begin{cases}
	B^{r,s}_A \otimes B^{2n-r,s}_A & \text{if $r<n$,}\\
	B^{n,s}_A & \text{if $r=n$ for type $D_{n+1}^{(2)}$,}\\
	\left(B^{n,s}_A\right)^{\otimes 2} & \text{if $r=n$ for type $A_{2n}^{(2)}$, $A^{(2)\dagger}_{2n}$,}\\
	B^{n,2s}_A & \text{if $r=n$ for type $C_n^{(1)}$.}
	\end{cases}
\end{equation}
It was shown in~\cite[Thm.~5.1]{Okado.2013} that these crystals are aligned, proving~\cite[Conj. 6.6]{OSS03III}.
Hence they give realizations for the corresponding KR crystals.

For types $B_n^{(1)}$ and $A_{2n-1}^{(2)}$, we use the realizations
\begin{equation}
\label{equation.embedding D}
\virtual{B}^{r,s} = \begin{cases}
	B^{r,s}_D & \text{if $r<n$ for type $A_{2n-1}^{(2)}$}\\
	B^{r,2s}_D & \text{if $r<n$ for type $B_n^{(1)}$,}\\
	B^{n,s}_D \otimes B^{n+1,s}_D & \text{if $r=n$ for type $A_{2n-1}^{(2)}$, $B_n^{(1)}$.}
	\end{cases}
\end{equation}
\end{subequations}
The realization of the first two lines are given in~\cite[Thm. 5.14]{schilling.scrimshaw.2015}.
The last realization is given in~\cite[Thm.~5.1(2-ii)]{Okado.2013}.

For $B = \bigotimes_{k=1}^N B^{r_k,s_k}$, we will use the notation $\virtual{B} = \bigotimes_{k=1}^N \virtual{B}^{r_k,s_k}$.

%%%%%%%%%%%%%%%%%%%%%%%%%%%%%%%%%%%%%%%%%%%%%%%%%%%
\subsection{Rigged configurations}

Fix a tensor product of KR crystals $B$.
Let $L(B) := (L_s^{(r)})_{r \in I_0, s \in \ZZ_{>0}}$, where $L_s^{(r)}$ equals the number of tensor factors 
$B^{r,s}$ occurring in $B$. When $B$ is clear, we denote this simply by $L$.

A \defn{rigged configuration} $(\nu, J) \in \rc(L)$ is a sequence of partitions $(\nu^{(a)})_{a \in I_0}$, where to each
row of $\nu^{(a)}$ we associate an integer\footnote{For type $A_{2n}^{(2)\dagger}$, we require riggings for odd length 
rows of $\nu^{(n)}$ to instead be in $\ZZ + 1/2$.} $x$, called \defn{rigging}. The pair $(i,x)$, where $i$ is the length of a row
and $x$ is the associated rigging, is called a \defn{string}.
Let $J_i^{(a)}$ denote the multiset of riggings of rows of length $i$ in $\nu^{(a)}$.
The rigging $x \in J_i^{(a)}$ needs to satisfy the condition
\[
	0 \leqslant x \leqslant p_i^{(a)}(\nu; L),
\]
where $p_i^{(a)}(\nu; L)$ is the \defn{vacancy number} 
\[
	p_i^{(a)}(\nu; L) = \sum_{j \in \ZZ_{>0}} L_j^{(a)} \min(i, j) - \sum_{b \in I_0} \frac{A_{ab}}{\gamma_b} 
	\sum_{j \in \ZZ_{>0}} \min(\gamma_a i, \gamma_b j) m_j^{(b)}
\]
except when $\g$ is of type $A_{2n}^{(2)}$ or $A_{2n}^{(2)\dagger}$, where
\[
	p_i^{(a)}(\nu; L) = \sum_{j \in \ZZ_{>0}} L_j^{(a)} \min(i, j) - \sum_{b \in I_0} \frac{A_{ab}}{1 + \delta_{x n}} 
	\sum_{j \in \ZZ_{>0}} \min(i, j) m_j^{(b)},
\]
with $x = b, a$ for type $A_{2n}^{(2)}, A_{2n}^{(2)\dagger}$ respectively.
If $\g$ is of simply-laced type, we set $\gamma_a = 1$ for all $a \in I$.
Here $m_i^{(a)}$ is the number of parts of size $i$ in $\nu^{(a)}$.
When $\nu$ and $L$ (or $B$) are clear from context, we simply write $p_i^{(a)}$.
As usual for partitions, we identify two rigged partitions $\nu^{(a)}$ and $\widetilde{\nu}^{(a)}$ if their parts
(with their riggings) are permuted.
Let $(\nu, J)^{(a)}$ be the partition $\nu^{(a)}$ with all of its associated riggings.
We say a string of $(\nu, J)^{(a)}$ is \defn{singular} if it is of the form $(i, p_i^{(a)})$.

\begin{example}
\label{ex:rc_example}
Let $B = B^{2,4} \otimes B^{1,2} \otimes B^{5,1} \otimes B^{3,2}$ in type $C_5^{(1)}$.
The following is a rigged configuration in $\rc(L(B), \clfw_1 + \clfw_2 + \clfw_4)$:
\[
(\nu, J) =
\begin{tikzpicture}[scale=0.29, baseline=-20]
\rpp{5,1}{0,0}{1,0}
\begin{scope}[xshift=8cm]
\rpp{5,4,2}{0,1,0}{1,1,1}
\end{scope}
\begin{scope}[xshift=16cm]
\rpp{5,4,2,2}{0,0,0,0}{0,0,0,0}
\end{scope}
\begin{scope}[xshift=24cm]
\rpp{5,4,2,2}{0,0,0,0}{0,0,0,0}
\end{scope}
\begin{scope}[xshift=32cm]
\rpp{3,2,1,1}{0,1,1,0}{0,1,1,1}
\end{scope}
\end{tikzpicture}.
\]
Here the vacancy number $p_i^{(a)}$ is written to the left of each part of length $i$ in the partition $\nu^{(a)}$ and
the riggings in $J_i^{(a)}$ appear to the right of the parts of length $i$ in $\nu^{(a)}$.
\end{example}

\begin{remark}
The rigged configurations that we give here differ slightly from those given in~\cite{OSS:2003a}.
In particular, for type $B_n^{(1)}$ (resp.\ type $C_n^{(1)}$), our rigged configurations use full-width boxes instead 
of half-width boxes (resp.\ double width boxes) for $\nu^{(n)}$. To go to the rigged configurations of~\cite{OSS:2003a}, 
simply half (resp.\ double) the partition $\nu^{(n)}$ for type $B_n^{(1)}$ (resp.\ type $C_n^{(1)}$).
\end{remark}

Note that for $i \gg 1$, we have $p_i^{(a)} = p_{i+1}^{(a)}$, and hence $p_{\infty}^{(a)} := p_i^{(a)}$ for some $i \gg 1$.
This can also be seen by directly substituting $i = \infty$, which results in $\min(\infty, j) = j$.
Define the weight of a rigged configuration
\[
\wt(\nu, J) = k_0 \Lambda_0 + \sum_{a \in I_0} \kappa_a p_{\infty}^{(a)} \Lambda_a,
\]
where $k_0$ is such that $\langle c, \wt(\nu, J) \rangle = 0$.
There exists an extension of rigged configurations to a $U_q(\g_0)$-crystal that was given 
in~\cite{Sch:2006, schilling.scrimshaw.2015}.
When restricting to a given weight space, we write
\[
	\rc(L,\lambda) = \{(\nu, J) \in \rc(L) \mid \wt_0(\nu, J) = \lambda\},
\]
where $\wt_0(\nu, J)$ is the $\g_0$-weight.
The \defn{complement rigging} involution $\theta \colon \rc(L) \to \rc(L)$ is defined by replacing every rigging 
$x \in J_i^{(a)}$ with its corresponding \defn{corigging} $p_i^{(a)} - x$.

Rigged configurations are known to be well-behaved under the embeddings given 
by~\eqref{eq:ambient_emb}~\cite{OSS03III, OSS03II, schilling.scrimshaw.2015}.
We can explicitly construct the embedding $\emb \colon \rc(L) \to \vrc(\virtual{L})$, where $\virtual{L}(B) := L(\virtual{B})$
and the embedded rigged configuration $(\virtual{\nu}, \virtual{J}) := \emb(\nu, J)$ is given by
\begin{subequations}
\label{eq:virtual_RC}
\begin{align}
\label{eq:virtual_m} 
	\virtual{m}_{\gamma_a i}^{(b)} & = m_i^{(a)},\\ 
\label{eq:virtual_J} 
	\virtual{J}_{\gamma_a i}^{(b)} & = \gamma_a J_i^{(a)},
\end{align}
\end{subequations}
where $\virtual{m}_j^{(b)} = 0$ when $j \notin \gamma_a \ZZ_{>0}$,
for all $b \in \phi^{-1}(a)$, except for $a=n$ in types $A_{2n}^{(2)}$ and $A_{2n}^{(2)\dagger}$, where
\begin{subequations}
\label{eq:virtual_RC_A2}
\begin{align}
	\virtual{m}_{i}^{(n)} & = m_i^{(n)},\\ 
	\virtual{J}_i^{(n)} & = 2 J_i^{(n)}.
\end{align}
\end{subequations}
Moreover, this completely characterizes $\emb\bigl(\rc(L)\bigr)$.
Also note that
\begin{equation}
\label{eq:virtual_vacancy}
\virtual{p}_i^{(b)} = \begin{cases}
2 p_i^{(n)} & \text{if } a = n \text{ for type } A_{2n}^{(2)}, A_{2n}^{(2)\dagger}, \\
\gamma_a p_i^{(a)} & \text{otherwise,}
\end{cases}
\end{equation}
for all $a \in I_0$, $i \in \ZZ_{>0}$ and $b \in \phi^{-1}(a)$.

%%%%%%%%%%%%%%%%%%%%%%%%%%%%%%%%%%%%%%%%%%%%%%%%%%%
\subsection{The bijection $\Phi$ for types $A_n^{(1)}$ and $D_n^{(1)}$}
\label{subsection.bij AD}

Let $\g$ be of type $A_n^{(1)}$ or $D_n^{(1)}$, and let $B$ be a tensor product of KR crystals.
The set of \defn{paths} corresponding to $B$ consists of all $I_0$-highest weight elements in $B$:
\[
	\p(B) = \{ b \in B \mid e_i b=0 \text{ for $i\in I_0$}\}.
\]
When also restricting to a given weight space, we write
\[
	\p(B,\lambda) = \{b \in \p(B) \mid \wt_0(b) = \lambda\},
\]
where $\wt_0$ is the $\g_0$-weight.

For type $A_n^{(1)}$, a bijection $\Phi \colon \p(B) \to \rc(L)$ was established in~\cite{KSS:2002}, generalizing ideas 
in~\cite{KKR:1986,KR86}. Moreover, the paths $\p(B)$ are described using the usual semistandard tableaux whose 
entries are in $\{1,2, \dotsc, n+1\}$. For type $D_n^{(1)}$, a bijection $\Phi \colon \p(B) \to \rc(L)$ was established 
in~\cite{S:2005,OSS:2013,OSSS16}. In this case, the KR crystals are described in terms of \defn{Kirillov--Reshetikhin 
(KR) tableaux}~\cite{S:2005,OSS:2013} rather than Kashiwara--Nakashima tableaux~\cite{KN:1994} that were previously 
used to describe $B^{r,s}$~\cite{FOS:2009}. In the KR tableau formulation, the elements 
in $B^{r,s}$ are represented by rectangular tableaux of width $s$ and height $r$ with entries in $B^{1,1}$.

For both types $A_n^{(1)}$ and $D_n^{(1)}$, the bijection $\Phi \colon \p(B) \to \rc(L)$ is defined recursively.
Let $B^{\bullet}$ be a tensor product of KR crystals.
On the path side, the composition of the following maps is used:
\begin{align*}
	\lh \colon \p(B^{1,1} \otimes B^{\bullet}) & \to \p(B^{\bullet}), &&\\ 
	\young(b) \otimes b^{\bullet} & \mapsto b^{\bullet}, &&\\[3mm]
	\lhs \colon \p(B^{r,1} \otimes B^{\bullet}) & \to \p(B^{\bullet}) && \hspace{-0.8cm}(r = n-1, n, \; \g = D_n^{(1)}) \\
	b \otimes b^{\bullet} & \mapsto b^{\bullet}, &\\[3mm]
	\lb \colon \p(B^{r,1} \otimes B^{\bullet}) & \to \p(B^{1,1} \otimes B^{r-1,1} \otimes B^{\bullet}) && \hspace{-1.1cm}
	 \left({\arraycolsep=1.5pt \begin{array}{cl} 1 < r \leqslant n, & \g = A_n^{(1)} \\ 1 < r < n-1, & \g = D_n^{(1)} \end{array}} \!\right),\\
	 \begin{array}{|c|} \hline b_1 \\\hline \vdots \\\hline b_{r-1} \\\hline b_r \\\hline \end{array} \otimes b^{\bullet} & 
	 \mapsto 
	\begin{array}{|c|} \hline b_r \\\hline \end{array} \otimes \begin{array}{|c|} \hline b_1 \\\hline \vdots \\\hline b_{r-1} \\
	\hline \end{array} \otimes b^{\bullet}, &&\\[3mm]
	 \ls \colon \p(B^{r,s} \otimes B^{\bullet}) & \to \p(B^{r,1} \otimes B^{r,s-1} \otimes B^{\bullet}) && (s \geqslant 2),\\
	\begin{array}{|c|c|c|c|} \hline b_{11} & b_{12} & \cdots & b_{1s} \\\hline \vdots & \vdots & \ddots & \vdots \\\hline b_{r1} 
	& b_{r2} & \cdots & b_{rs} \\\hline \end{array} \otimes b^{\bullet} & \mapsto \begin{array}{|c|} \hline b_{11} \\\hline \vdots 
	\\ \hline b_{r1} \\\hline \end{array} \otimes  \begin{array}{|c|c|c|} \hline b_{12} & \cdots & b_{1s} \\\hline \vdots & \ddots 
	& \vdots \\ \hline b_{r2} & \cdots & b_{rs} \\\hline \end{array} \otimes b^{\bullet},
	&& 
\end{align*}
until the empty path is reached. 

The analogous maps on the rigged configuration side are denoted by:
\begin{subequations}
\label{equation.rc maps}
\begin{align}
	\delta \colon \rc(L) & \to \rc\bigl(\lh(L)\bigr),\\
	\deltas \colon \rc(L) & \to \rc\bigl(\lhs(L)\bigr) && (r = n-1, n, \; \g = D_n^{(1)}),\\
	\beta \colon \rc(L) & \to \rc\bigl(\lb(L)\bigr) 
	&&\left(\begin{array}{cc} 1 < r \leqslant n, & \g = A_n^{(1)} \\ 1 < r < n-1, & 
	\g = D_n^{(1)}\end{array} \right),\\
	\gamma \colon \rc(L) & \to \rc\bigl(\ls(L)\bigr) && (s \geqslant 2),
\end{align}
\end{subequations}
where $\lx\bigl(L(B)\bigr) := L\bigl(\lx(B)\bigr)$ for $\mathrm{x} = \mathrm{h}, \mathrm{h}_{\mathrm{sp}},
\mathrm{b, s}$.

The map $\delta \colon \rc(L) \to \rc\bigl(\lh(L)\bigr)$ can be constructed by the following recursive 
procedure. Start with the highest weight crystal element $b_0 = u_{\clfw_1}$ and set $\ell_0 = 1$.
Suppose $b_i$ and $\ell_i$ have been constructed for $0\leqslant i <j$. Consider step $j$.
Let $\ell_j$ denote a minimal $i_a \geqslant \ell_{j-1}$ for all $a \in I_0$ such that $f_a b_{j-1} \neq 0$ and 
$(\nu,J)^{(a)}$ has a singular string of length $i_a$ that has not been previously selected.\footnote{If there are multiple 
such $a \in I_0$, then the resulting rigged configuration is independent of this choice.}
If no such string exists, terminate, set all $\ell_{j'} = \infty$ for $j' \geqslant j$ and return $b_{j-1}$.
Otherwise set $b_j = f_a b_{j-1}$ and $\ell_j=i_a$ and repeat.
We form a new rigged configuration by removing a box from each string selected by the algorithm above, 
making the resulting rows singular and keeping all other strings the same.

The map $\deltas \colon \rc(L) \to \rc\bigl(\lhs(L)\bigr)$ is analogous to $\delta$ but using 
$B(\clfw_r)$, for $r = n-1, n$, instead of $B(\clfw_1)$.

We note that $\delta$ was first described for type $A_n^{(1)}$ in~\cite{KR86, KKR:1986} and for type
 $D_n^{(1)}$ in~\cite{OSS:2003a}.
The map $\delta$ for the other nonexceptional types was described in~\cite{OSS:2003a}.
 The map $\deltas$ as given here was first described in~\cite{OS12, Scrimshaw17}, where it was shown to be equivalent 
 to the definition given in~\cite{S:2005}.
 
The map $\beta$ is given for all nonexceptional types by adding a length 1 singular string to $(\nu,J)^{(a)}$ for all $a < r$.
The map $\gamma$ is the identity map, but we note that some vacancy numbers will change.
In the sequel, we will need to indicate the left factor when applying $\gamma$, and so we denote this by 
$\gamma_{r,s} \colon \rc\bigl(L(B^{r,s} \otimes B^{\bullet})\bigr) \to \rc\bigl(L(B^{r,1} \otimes B^{r,s-1} \otimes B^{\bullet})\bigr)$.

We also have right analogs of the maps given above.
Let $\mathbin{\uparrow}(b)$ denote the map which takes an element $b \in B(\lambda) \subset B$ to the 
corresponding $I_0$-highest weight element $u_{\lambda} \in B(\lambda)$. We define
\begin{align*}
\widetilde{\delta} & := \theta \circ \delta \circ \theta,
&
\rh & := \mathbin{\uparrow} \circ \lusz \circ \lh \circ \lusz = \hwstar \circ \lh \circ \hwstar,
\\
\deltast & := \theta \circ \deltas \circ \theta,
&
\rh_{\mathrm{sp}} & := \lusz \circ \lhs \circ \lusz = \hwstar \circ \lhs \circ \hwstar,
\\
\widetilde{\beta} & := \theta \circ \beta \circ \theta,
&
\rb & := \lusz \circ \lb \circ \lusz = \hwstar \circ \lb \circ \hwstar,
\\
\widetilde{\gamma} & := \theta \circ \gamma \circ \theta,
&
\rs & := \lusz \circ \ls \circ \lusz = \hwstar \circ \ls \circ \hwstar,
\end{align*}
where $\hwstar := \mathbin{\uparrow} \circ \lusz$.

\begin{example}
Let $B = B^{4,1} \otimes B^{2,2}$ in type $D_4^{(1)}$.
We begin with a rigged configuration in $\rc(L(B), \clfw_4)$ and the corresponding element under $\Phi$ and apply the 
following sequence of maps:
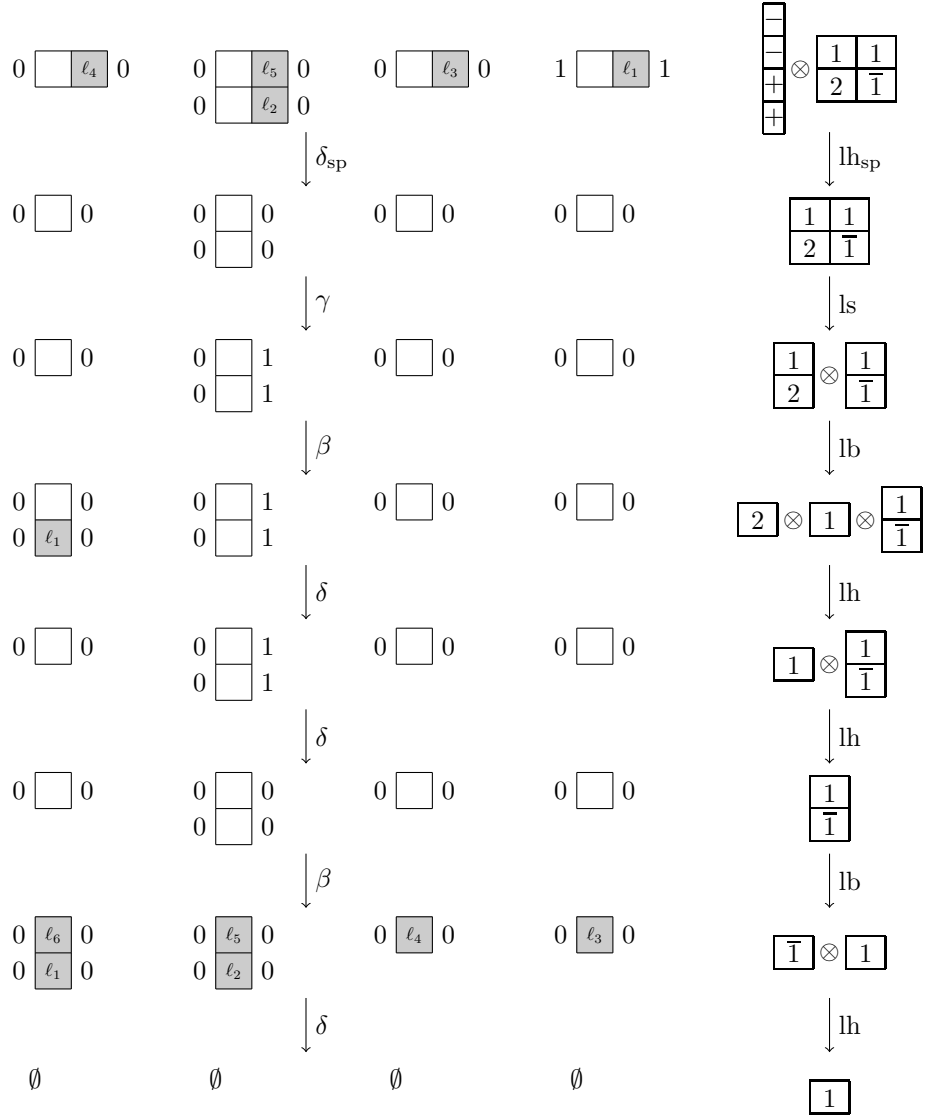
\[
\def\lr#1{\raisebox{-.3ex}{$#1$}}  % To make the printing a little better spaced
\begin{tikzpicture}[scale=0.48]
	\fill[black!20] (16,-1) rectangle (17,-2);
	\node[scale=.7] at (16.5, -1.5) {$\ell_1$};
	\fill[black!20] (6,-2) rectangle (7,-3);
	\node[scale=.7] at (6.5, -2.5) {$\ell_2$};
	\fill[black!20] (11,-1) rectangle (12,-2);
	\node[scale=.7] at (11.5, -1.5) {$\ell_3$};
	\fill[black!20] (1,-1) rectangle (2,-2);
	\node[scale=.7] at (1.5, -1.5) {$\ell_4$};
	\fill[black!20] (6,-1) rectangle (7,-2);
	\node[scale=.7] at (6.5, -1.5) {$\ell_5$};
	\rpp{2}{0}{0}
	\begin{scope}[xshift=5cm]
	\rpp{2,2}{0,0}{0,0}
	\end{scope}
	\begin{scope}[xshift=10cm]
	\rpp{2}{0}{0}
	\end{scope}
	\begin{scope}[xshift=15cm]
	\rpp{2}{1}{1}
	\end{scope}
\draw[->] (7.5,-3.25) -- (7.5,-4.75) node[midway,right] {$\deltas$};
\begin{scope}[yshift=-4cm]
	\rpp{1}{0}{0}
	\begin{scope}[xshift=5cm]
	\rpp{1,1}{0,0}{0,0}
	\end{scope}
	\begin{scope}[xshift=10cm]
	\rpp{1}{0}{0}
	\end{scope}
	\begin{scope}[xshift=15cm]
	\rpp{1}{0}{0}
	\end{scope}
\end{scope}
\draw[->] (7.5,-7.25) -- (7.5,-8.75) node[midway,right] {$\gamma$};
\begin{scope}[yshift=-8cm]
	\rpp{1}{0}{0}
	\begin{scope}[xshift=5cm]
	\rpp{1,1}{1,1}{0,0}
	\end{scope}
	\begin{scope}[xshift=10cm]
	\rpp{1}{0}{0}
	\end{scope}
	\begin{scope}[xshift=15cm]
	\rpp{1}{0}{0}
	\end{scope}
\end{scope}
\draw[->] (7.5,-11.25) -- (7.5,-12.75) node[midway,right] {$\beta$};
\begin{scope}[yshift=-12cm]
	\fill[black!20] (0,-2) rectangle (1,-3);
	\node[scale=.7] at (0.5, -2.5) {$\ell_1$};
	\rpp{1,1}{0,0}{0,0}
	\begin{scope}[xshift=5cm]
	\rpp{1,1}{1,1}{0,0}
	\end{scope}
	\begin{scope}[xshift=10cm]
	\rpp{1}{0}{0}
	\end{scope}
	\begin{scope}[xshift=15cm]
	\rpp{1}{0}{0}
	\end{scope}
\end{scope}
\draw[->] (7.5,-15.25) -- (7.5,-16.75) node[midway,right] {$\delta$};
\begin{scope}[yshift=-16cm]
	\rpp{1}{0}{0}
	\begin{scope}[xshift=5cm]
	\rpp{1,1}{1,1}{0,0}
	\end{scope}
	\begin{scope}[xshift=10cm]
	\rpp{1}{0}{0}
	\end{scope}
	\begin{scope}[xshift=15cm]
	\rpp{1}{0}{0}
	\end{scope}
\end{scope}
\draw[->] (7.5,-19.25) -- (7.5,-20.75) node[midway,right] {$\delta$};
\begin{scope}[yshift=-20cm]
	\rpp{1}{0}{0}
	\begin{scope}[xshift=5cm]
	\rpp{1,1}{0,0}{0,0}
	\end{scope}
	\begin{scope}[xshift=10cm]
	\rpp{1}{0}{0}
	\end{scope}
	\begin{scope}[xshift=15cm]
	\rpp{1}{0}{0}
	\end{scope}
\end{scope}
\draw[->] (7.5,-23.25) -- (7.5,-24.75) node[midway,right] {$\beta$};
\begin{scope}[yshift=-24cm]
	\fill[black!20] (0,-2) rectangle (1,-3);
	\node[scale=.7] at (0.5, -2.5) {$\ell_1$};
	\fill[black!20] (5,-2) rectangle (6,-3);
	\node[scale=.7] at (5.5, -2.5) {$\ell_2$};
	\fill[black!20] (15,-1) rectangle (16,-2);
	\node[scale=.7] at (15.5, -1.5) {$\ell_3$};
	\fill[black!20] (10,-1) rectangle (11,-2);
	\node[scale=.7] at (10.5, -1.5) {$\ell_4$};
	\fill[black!20] (5,-1) rectangle (6,-2);
	\node[scale=.7] at (5.5, -1.5) {$\ell_5$};
	\fill[black!20] (0,-1) rectangle (1,-2);
	\node[scale=.7] at (0.5, -1.5) {$\ell_6$};
	\rpp{1,1}{0,0}{0,0}
	\begin{scope}[xshift=5cm]
	\rpp{1,1}{0,0}{0,0}
	\end{scope}
	\begin{scope}[xshift=10cm]
	\rpp{1}{0}{0}
	\end{scope}
	\begin{scope}[xshift=15cm]
	\rpp{1}{0}{0}
	\end{scope}
\end{scope}
\draw[->] (7.5,-27.25) -- (7.5,-28.75) node[midway,right] {$\delta$};
\begin{scope}[yshift=-29.5cm]
	\node at (0,0) {$\emptyset$};
	\begin{scope}[xshift=5cm]
	\node at (0,0) {$\emptyset$};
	\end{scope}
	\begin{scope}[xshift=10cm]
	\node at (0,0) {$\emptyset$};
	\end{scope}
	\begin{scope}[xshift=15cm]
	\node at (0,0) {$\emptyset$};
	\end{scope}
\end{scope}
\begin{scope}[xshift=22cm]
\node at (0,-1.5) {$\begin{array}{|@{}c@{}|} \hline - \\\hline - \\\hline + \\\hline + \\\hline \end{array} \otimes \begin{array}{|c|c|} \hline \lr{1} & \lr{1} \\\hline \lr{2} & \lr{\mone} \\\hline \end{array}$};
\draw[->] (0,-3.25) -- (0,-4.75) node[midway,right] {$\lhs$};
\node at (0,-6) {$\begin{array}{|c|c|} \hline \lr{1} & \lr{1} \\\hline \lr{2} & \lr{\mone} \\\hline \end{array}$};
\draw[->] (0,-7.25) -- (0,-8.75) node[midway,right] {$\ls$};
\node at (0,-10) {$\begin{array}{|c|} \hline \lr{1} \\\hline \lr{2} \\\hline \end{array} \otimes \begin{array}{|c|c|} \hline \lr{1} \\\hline  \lr{\mone} \\\hline \end{array}$};
\draw[->] (0,-11.25) -- (0,-12.75) node[midway,right] {$\lb$};
\node at (0,-14) {$\begin{array}{|c|} \hline \lr{2} \\\hline \end{array} \otimes \begin{array}{|c|} \hline \lr{1} \\\hline \end{array} \otimes \begin{array}{|c|c|} \hline \lr{1} \\\hline  \lr{\mone} \\\hline \end{array}$};
\draw[->] (0,-15.25) -- (0,-16.75) node[midway,right] {$\lh$};
\node at (0,-18) {$\begin{array}{|c|} \hline \lr{1} \\\hline \end{array} \otimes \begin{array}{|c|c|} \hline \lr{1} \\\hline  \lr{\mone} \\\hline \end{array}$};
\draw[->] (0,-19.25) -- (0,-20.75) node[midway,right] {$\lh$};
\node at (0,-22) {$\begin{array}{|c|} \hline \lr{1} \\\hline \lr{\mone} \\\hline \end{array}$};
\draw[->] (0,-23.25) -- (0,-24.75) node[midway,right] {$\lb$};
\node at (0,-26) {$\begin{array}{|c|} \hline \lr{\mone} \\\hline \end{array} \otimes \begin{array}{|c|} \hline \lr{1} \\\hline \end{array}$};
\draw[->] (0,-27.25) -- (0,-28.75) node[midway,right] {$\lh$};
\node at (0,-30) {$\begin{array}{|c|} \hline \lr{1} \\\hline \end{array}$};
\end{scope}
\end{tikzpicture}
\]
\end{example}

We need an additional operation $\lb^{(s)}$ on $\p(B)$ and $\beta^{(s)}$ on $\rc(L)$ for type $D^{(1)}_n$.

\begin{definition}
\label{defn:lbs}
For $1 < r \leqslant n-2$, the map $\lb^{(s)} \colon \p(B^{r,s} \ot B^\bullet)\rightarrow\p(B^{1,s}\ot B^{r-1,s}\ot B^\bullet)$
is defined by 
\[
\lb^{(s)}(b\ot b^\bullet)=b'\ot b''\ot b^\bullet.
\]
For $b=u_\la$,  the $I_0$-highest weight element of weight $\lambda = (s-N)\clfw_r + \clfw_{r-k_N} + \cdots + \clfw_{r-k_1}$
with $0 < k_N \leqslant \cdots \leqslant k_1$, $b'\ot b''\in B^{1,s}\ot B^{r-1,s}$ is given by
\[
\lb^{(s)}(u_{\lambda}) = {
\def\lr#1{\raisebox{-.3ex}{$#1$}}  % To make the printing a little better spaced
\begin{array}{|c|c|c|c|c|c|} \hline \lr{r} & \lr{\cdots} & \lr{r} & \lr{\overline{r - k_N + 1}} & \lr{\cdots} & 
\lr{\overline{r - k_1 + 1}} \\\hline \end{array} \otimes u_{\lambda_-},
}
\]
where
\[
\lambda_- = (s-N) \clfw_{r-1} + \clfw_{r-k_N+1} + \cdots + \clfw_{r-k_1+1}.
\]
For other elements $b$, the map $\lb^{(s)}$ is determined 
by assuming that $\lb^{(s)}$ commutes with $f_i$  for $i\in I_0$. 
\end{definition}

\begin{definition}
\label{defn:bs}
For a rigged configuration $(\nu,J) \in \rc\bigl(L(B^{r,s}\ot B^\bullet)\bigr)$,
the operation $\beta^{(s)}$ adds a singular string of length $s$ to $\nu^{(a)}$ for all $1 \leqslant a < r$.
\end{definition}

\begin{proposition}
\label{prop:betas_properties}
The operation $\beta^{(s)}$ preserves vacancy numbers and
\[
[\beta^{(s)}, \widetilde{\delta}] = [\beta^{(s)}, \deltast] = [\beta^{(s)}, \widetilde{\beta}] = [\beta^{(s)}, \widetilde{\gamma}] = 0.
\]
\end{proposition}

\begin{proof}
It is straightforward to see that $\beta^{(s)}$ preserves vacancy numbers.
It immediately follows that
\[
[\beta^{(s)}, \widetilde{\beta}] = [\beta^{(s)}, \widetilde{\gamma}] = 0.
\]
We show $[\beta^{(s)}, \widetilde{\delta}] = 0$ by a similar argument to~\cite[Prop.~3.16]{OSSS16}, which is 
based on~\cite[Lemma~5.4]{S:2005}.
Let $(\nu, J) \in \rc\bigl(L(B)\bigr)$, where $B = B^{r,s} \otimes B^{\bullet} \otimes B^{1,1}$.
Let $\ell^{(a)}$ and $\overline{\ell}^{(a)}$ (resp.~$k^{(a)}$ and $\overline{k}^{(a)}$) be the cosingular strings, that is,
rows with rigging $0$, selected by $\widetilde{\delta}$ in $(\nu, J)$ (resp.~$\beta^{(s)}(\nu, J)$).
Since $\beta^{(s)}$ preserves vacancy numbers and riggings, we must have $k^{(a)} \leqslant \ell^{(a)}$ and 
$\overline{k}^{(a)} \leqslant \overline{\ell}^{(a)}$.
If $k^{(a)} = \ell^{(a)}$ and $\overline{k}^{(a)} = \overline{\ell}^{(a)}$ for all $a \in I_0$, the claim follows.
We show $k^{(a)} = \ell^{(a)}$ as the case for $\overline{k}^{(a)} = \overline{\ell}^{(a)}$ is similar.

Assume $b$ is minimal such that $k^{(b)} < \ell^{(b)}$, where $1 \leqslant b < r$.
Therefore, we must have $k^{(a)} = \ell^{(a)} \leqslant s$ for all $1 \leqslant a < b$, $k^{(b)} = s < \ell^{(b)}$, 
$m_s^{(b)} = 0$ and $p_s^{(b)} = 0$. Note that
\begin{equation}
\label{eq:convexity}
-p_{i-1}^{(a)} + 2 p_i^{(a)} - p_{i+1}^{(a)} = -\sum_{b \in I_0} A_{ab} m_i^{(b)} + L_i^{(a)}.
\end{equation}
Since $m_s^{(b)} = 0$ and $p_s^{(b)} = 0$, by~\eqref{eq:convexity} at $i = s$, we have
\begin{equation}
\label{eq:convexity_is}
-p_{s-1}^{(b)} - p_{s+1}^{(b)} = m_s^{(b-1)} + m_s^{(b+1)} + L_s^{(b)}.
\end{equation}
Since $p_i^{(a)} \geqslant 0$, this implies all quantities in~\eqref{eq:convexity_is} must be zero.
For $b > 1$, we must have $k^{(b-1)} = \ell^{(b-1)} \leqslant s-1$ as $\ell^{(b-1)} = s$ contradicts $m_s^{(b-1)} = 0$.
Recall that $p_{s-1}^{(b)} = 0$, and so any row of length $s-1$ must be cosingular since all riggings are nonnegative.
Therefore, we must have $m_{s-1}^{(b)} = 0$ as otherwise $\ell^{(b)} = s-1$, but this implies that $m_{s-1}^{(b-1)} = 0$ since
\[
-p_{s-2}^{(b)} = m_{s-1}^{(b-1)} + m_{s-1}^{(b+1)} + L_{s-1}^{(b)}
\]
from~\eqref{eq:convexity} at $i = s-1$.
Hence, we must have $k^{(b-1)} = \ell^{(b-1)} \leqslant s-2$.
Therefore, by iterating the previous argument, we can show that $m_i^{(b-1)} = 0$ for all $i \leqslant s$, which is 
a contradiction.
For $b = 1$, similar to the previous case, we must have $m_i^{(1)} = 0$ and $p_i^{(1)} = 0$ for all $1 \leqslant i \leqslant s$.
Moreover, we must have $L_1^{(1)} = 0$, but this is a contradiction.
Therefore, we have $[\beta^{(s)}, \widetilde{\delta}] = 0$.

Let $\emb_{2\times}$ be the doubling map:
the virtualization (or similarity) map given by $\virtual{\g} = \g$ and $\gamma_a = 2$ for all $a \in I$~\cite{K96}.
Recall from~\cite{S:2005, Scrimshaw17}
\[
\deltas = \emb_{2\times}^{-1} \circ (\delta \circ \beta)^{n-2} \circ \delta \circ \overline{\beta} 
\circ \delta \circ \beta_{[r]} \circ \emb_{2\times},
\]
where $\beta_{[r]}$ (resp.~$\overline{\beta}$) adds a singular row of length 1 to $\nu^{(a)}$ for all $a \in I_0 \setminus \{r\}$ 
(resp.~$a \leqslant n - 2$).
Define $\widetilde{\overline{\beta}} := \theta \circ \overline{\beta} \circ \theta$ and 
$\widetilde{\beta}_{[r]} := \theta \circ \beta_{[r]} \circ \theta$.
It is clear that $\emb_{2\times} \circ \beta^{(s)} = \beta^{(2s)} \circ \emb_{2\times}$ and
\[
[\beta^{(s)}, \widetilde{\overline{\beta}}] = [\beta^{(s)}, \widetilde{\beta}_{[r]}] = [\emb_{2\times}, \theta] = 0
\]
from~\cite[Lemma~3.15]{OSSS16}.
Since
\[
\deltast = \emb_{2\times}^{-1} \circ (\widetilde{\delta} \circ \widetilde{\beta})^{n-2} \circ \widetilde{\delta} 
\circ \widetilde{\overline{\beta}} \circ \widetilde{\delta} \circ \widetilde{\beta}_{[r]} \circ \emb_{2\times},
\]
we have $[\beta^{(s)}, \deltast] = 0$.
\end{proof}

\begin{proposition}
\label{prop:general_lb_beta}
For type $D^{(1)}_n$, the operation $\lb^{(s)}$ corresponds to $\beta^{(s)}$ under $\Phi$.
\end{proposition}

\begin{proof}
By Proposition~\ref{prop:betas_properties}, it suffices to consider the $I_0$-highest weight elements of $B = B^{r,s}$.

Consider an $I_0$-highest weight element $u_{\lambda} \in B^{r,s}$.
From~\cite[Prop.~3.3]{OSS:2013}, we have $\Phi(u_{\lambda})$ given by
\[
\nu^{(a)} = \begin{cases}
\overline{\lambda}^{[r-a]} & \text{if } 1 \leqslant a < r, \\
\overline{\lambda} & \text{if } r \leqslant a \leqslant n-2, \\
\overline{\lambda}^o & \text{if } a =n-1, n,
\end{cases}
\]
and all riggings are $0$, where $\overline{\lambda}$ is the complement partition of $\lambda$ in a $r \times s$ box, 
$\mu^{[k]}$ is the partition obtained by removing the $k$ longest rows of $\mu$ and $\mu^o$ is the partition 
consisting of only the odd rows of $\mu$. Suppose $\overline{\lambda}^t = (k_1, \dotsc, k_N)$.
Note that all strings in $\beta^{(s)}(\nu, J)$ are singular by the definition of $\beta^{(s)}$ and 
Proposition~\ref{prop:betas_properties}. Furthermore, the application of $\gamma$ makes all strings of length $i < s$ 
nonsingular in $\nu^{(1)}$, but keeps all other strings singular. Hence $\delta$ starts by selecting the singular string of 
length $s$ added by $\beta^{(s)}$ and continues selecting them until it reaches $\nu^{(r)}$, where we have two main cases.

We first consider the case $N < s$. In this case $\delta$ terminates and returns $r$, and similarly for each subsequent 
application of $\delta$ until the leftmost factor is $B^{1,N}$.
In this case, we continue selecting rows until we return again to the $r$-th partition, 
at which point there is a second singular string of length $N-1$ we can select since $k_N > 0$.
We continue selecting singular strings until we reach the $(r-k_N)$-th partition, 
at which point there are no other singular strings and we return $\overline{r-k_N+1}$.
The next application of $\delta$ is similar to the previous application except it terminates at the $(r-k_{N-1}+1)$-th partition
and returns $\overline{r-k_{N-1}+1}$. By iterating this argument, the claim follows.

Since the remaining factor is multiplicity free and determined by the weight, the claim follows.
\end{proof}

Define the involution $\varsigma \colon \rc(L) \to \rc(L^{\sigma})$ by $\varsigma\left((\nu, J)^{(a)}\right) = (\nu,J)^{(\sigma(a))}$.

\begin{proposition}
\label{prop:involution}
For types $A_n^{(1)}$ and $D_n^{(1)}$, the diagram
\[
\xymatrix{\p(B) \ar[r]^{\Phi} \ar[d]_{\sigma} & \rc(L) \ar[d]^{\varsigma} \\ \p(B^{\sigma}) \ar[r]_{\Phi} & \rc(L^{\sigma})}
\]
commutes.
\end{proposition}

\begin{proof}
For type $A_n^{(1)}$, this was shown in~\cite[Thm.~5.7]{OSS03III}.
Thus, let $\g$ be of type $D_n^{(1)}$.
Recall that $\sigma$ interchanges $n \leftrightarrow \mn$.
It is straightforward to see that $\sigma$ intertwines with $\lh$, $\lb$ and $\ls$.
Because $\beta$ and $\gamma$ affect neither $\nu^{(n-1)}$ nor $\nu^{(n)}$
(since $\gamma$ is the identity map and $\beta$ only applies to $\nu^{(a)}$ for $a<n-1$), 
the map $\varsigma$ intertwines with $\beta$ and $\gamma$. It remains to show that $\varsigma$
intertwines with $\delta$ and $\deltas$.

We first consider the case when $\delta$ does not terminate with $n$ or $\mn$.
In this case, the choices of $\nu^{(n-1)}$ and $\nu^{(n)}$ are done independently, and so the 
same rows are selected in $\varsigma(\nu, J) := (\nu_{\varsigma}, J_{\varsigma})$.
If $\delta$ terminates at $n$ or $\mn$, then the algorithm selected a row in $\nu^{(n-1)}$ or $\nu^{(n)}$, 
respectively, but not both. Therefore, we select the same row in $\nu_{\varsigma}^{(n)}$ or 
$\nu_{\varsigma}^{(n-1)}$, respectively, and terminate with $\mn$ or $n$, respectively.
Hence, the map $\varsigma$ intertwines with $\delta$.

For $\deltas$, we note that $\sigma \colon B^{n-1,1} \leftrightarrow B^{n,1}$.
From the description of $\deltas$, every $f_a$ arrow that would normally be applied
when doing $\deltas$, is replaced by an $f_{\sigma(a)}$ arrow.
However, we have $(\nu_{\varsigma}, J_{\varsigma})^{(a)} = (\nu,J)^{(\sigma(a))}$ and the claim follows.
\end{proof}

%%%%%%%%%%%%%%%%%%%%%%%%%%%%%%%%%%%%%%%%%%%%%%%%%%%
\section{The bijection}
\label{sec:bijection}
%%%%%%%%%%%%%%%%%%%%%%%%%%%%%%%%%%%%%%%%%%%%%%%%%%%

In this section, we show that there exists a bijection $\Phi$ from paths to rigged configurations for any nonexceptional type $\g$.
We construct $\Phi$ as a generalization of the map $\Phi \colon \p(B) \to \rc(L)$ of types $A_n^{(1)}$ and $D_n^{(1)}$
of Section~\ref{subsection.bij AD} to nonexceptional types using generalizations of the maps~\eqref{equation.rc maps} 
(the maps for paths remain the same).
In order to do this, we will \emph{define} $\Phi$ by using the map $\Phi^X$ in ambient type $X$ and lifted versions of 
the maps on paths and rigged configurations.
Afterwards, we will show that $\Phi$ can be computed by using generalizations of~\eqref{equation.rc maps}.

We begin by defining the generalizations of the rigged configuration operations~\eqref{equation.rc maps}.
The map $\delta$ in~\eqref{equation.rc maps} is defined in~\cite{OSS:2003a} and $\beta$ and $\gamma$ are the 
same maps as in~\eqref{equation.rc maps}. For type $D_{n+1}^{(2)}$, we define
\begin{subequations}
\label{eq:spins_Dtwisted}
\begin{align}
\deltas & := \emb_{2\times}^{-1} \circ \delta \circ (\delta \circ \beta)^{n-1} \circ \emb_{2\times},
\\ \lhs & := \emb_{2\times}^{-1} \circ \lh \circ (\lh \circ \lb)^{n-1} \circ \emb_{2\times}.
\end{align}
\end{subequations}
For type $B_n^{(1)}$, we define $\deltas$ and $\lhs$ by~\eqref{eq:spins_Dtwisted} except $\emb_{2\times}$ is the 
virtualization map with $\virtual{\g}$ of type $A_{2n-1}^{(2)}$ and $\gamma_a = 1 + \delta_{an}$ for all $a \in I$.
The KR tableaux used for $\p(B)$ are defined in~\cite{schilling.scrimshaw.2015} and precisely characterize the image 
of $\Phi$ under this bijection.

We also need the fact
\[
\emb\bigl(\p(B)\bigr) = \p\bigl(\emb(B)\bigr),
\]
where $B$ is a tensor product of KR crystals, which is a consequence of Proposition~\ref{prop:virtual_tensor_category} 
and $\emb\bigl(\p(B^{r,s})\bigr) = \p\bigl(\emb(B^{r,s})\bigr)$.

In this section, we continue with the notation of Section~\ref{sec:virtual_crystals}, where $f^X$ denotes a map $f$ in 
ambient type $X$.

%%%%%%%%%%%%%%%%%%%%%%%%%%%%%%%%%%%%%%%%%%%%%%%%%%%%%%%%%%
\subsection{Lifted operations on paths}
\label{subsection.lifted paths}

We define the operations $\virtual{\lh}$, $\lhsv$, $\virtual{\lb}$ and $\virtual{\ls}$ on paths for the ambient types.

%-------------------------------------------------------
\subsubsection{For ambient type $A$}

Let $B^{\bullet}$ be a tensor product of KR crystals of type $A_{2n-1}^{(1)}$.
We define $\virtual{\lh}$ as a composition of the following maps
\begin{equation}
\begin{split}
B_A^{1,1} \otimes B_A^{2n-1,1} \otimes B^{\bullet} & \xmap{\lh^A} B_A^{2n-1,1} \otimes B^{\bullet}
\xmap{\sigma} B_A^{1,1} \otimes  \sigma(B^{\bullet}) \\
& \xmap{\lh^{A}} \sigma(B^{\bullet}) \xmap{\sigma} B^{\bullet}.
\end{split}
\end{equation}
For $D^{(2)}_{n+1}$, the operation $\lhsv$ is defined by 
\begin{equation}
\lhsv := \lh^A \circ (\lh^A \circ \lb^A)^{n-1}.
\end{equation}

Next, we define $\virtual{\lb}$ for $1 < r < n$ in types $C_n^{(1)}$ and $D_{n+1}^{(2)}$
and $1 < r \leqslant n$ in types $A_{2n}^{(2)}$ and $A_{2n}^{(2)\dagger}$ by
\begin{subequations}
\label{eq:lbv_A}
\begin{equation}
\begin{split}
B_A^{r,1} \otimes B_A^{2n-r,1} \otimes B^{\bullet} & \xmap{\lb^A} B_A^{1,1} \otimes B_A^{r-1,1} \otimes B_A^{2n-r,1} \otimes B^{\bullet}
\\ & \xmap{R^A} B_A^{2n-r,1} \otimes B_A^{1,1} \otimes B_A^{r-1,1} \otimes B^{\bullet}
\\ & \xmap{\sigma\circ\lb^A\!\circ\sigma} B_A^{2n-1,1} \otimes B_A^{2n-r+1,1} \otimes B_A^{1,1} \otimes B_A^{r-1,1} \otimes B^{\bullet}
\\ & \xmap{R^A} B_A^{1,1} \otimes B_A^{2n-1,1} \otimes B_A^{r-1,1} \otimes B_A^{2n-r+1,1} \otimes B^{\bullet}.
\end{split}
\end{equation}
Recall that in type $D_{n+1}^{(2)}$, we use $\lhsv$ in the case $r=n$, so $\virtual{\lb}$ is not needed in this case.
For $r = n$ in type $C_n^{(1)}$, we define $\virtual{\lb}$ by
\begin{equation}
\label{eq:lbv_A_type_C}
\begin{split}
B_A^{n,2} \otimes B^{\bullet} & \xmap{\ls^A} B_A^{n,1} \otimes B_A^{n,1} \otimes B^{\bullet}
\\ & \xmap{\lb^A} B_A^{1,1} \otimes B_A^{n-1,1} \otimes B_A^{n,1} \otimes B^{\bullet}
\\ & \xmap{R^A} B_A^{n,1} \otimes B_A^{1,1} \otimes B_A^{n-1,1} \otimes B^{\bullet}
\\ & \xmap{\sigma\circ\lb^A\!\circ\sigma}  B_A^{2n-1,1} \otimes B_A^{n+1,1} \otimes B_A^{1,1} \otimes B_A^{n-1,1} \otimes B^{\bullet}
\\ & \xmap{R^A} B_A^{1,1} \otimes B_A^{2n-1,1} \otimes B_A^{n-1,1} \otimes B_A^{n+1,1} \otimes B^{\bullet}.
\end{split}
\end{equation}
\end{subequations}

Finally for $s>1$, we define $\virtual{\ls}$ for $1 \leqslant r < n$ in types $C_n^{(1)}$ and 
$D_{n+1}^{(2)}$ and $1 \leqslant r \leqslant n$ in types $A_{2n}^{(2)}$ and $A_{2n}^{(2)\dagger}$ by
\begin{subequations}
\label{eq:lsv_A}
\begin{equation}
\label{eq:lsv_A_generic}
\begin{split}
B_A^{r,s} \otimes B_A^{2n-r,s} \otimes B^{\bullet} & \xmap{\ls^A} B_A^{r,1} \otimes B_A^{r,s-1} \otimes B_A^{2n-r,s} \otimes B^{\bullet}
\\ & \xmap{R^A} B_A^{2n-r,s} \otimes B_A^{r,s-1} \otimes B_A^{r,1} \otimes B^{\bullet}
\\ & \xmap{\ls^A} B_A^{2n-r,1} \otimes B_A^{2n-r,s-1} \otimes B_A^{r,s-1} \otimes B_A^{r,1} \otimes B^{\bullet}
\\ & \xmap{R^A} B_A^{r,1} \otimes B_A^{2n-r,1} \otimes B_A^{r,s-1} \otimes B_A^{2n-r,s-1} \otimes B^{\bullet}.
\end{split}
\end{equation}
For $r = n$ in type $D_{n+1}^{(2)}$, we define
\begin{equation}
\virtual{\ls} := \ls^A.
\end{equation}
For $r = n$ in type $C_n^{(1)}$, we define $\virtual{\ls}$ by
\begin{equation}
\label{eq:lsv_A_Cn}
\begin{split}
B_A^{n,2s} \otimes B^{\bullet} & \xmap{\ls^A} B_A^{n,1} \otimes B_A^{n,2s-1} \otimes B^{\bullet}
\\ & \xmap{R^A} B_A^{n,2s-1} \otimes B_A^{n,1} \otimes B^{\bullet}
\\ & \xmap{\ls^A} B_A^{n,1} \otimes B_A^{n,2s-2} \otimes B_A^{n,1} \otimes B^{\bullet}
\\ & \xmap{R^A} B_A^{n,1} \otimes B_A^{n,1} \otimes B_A^{n,2s-2} \otimes B^{\bullet}
\\ & \xmap{(\ls^A)^{-1}} B_A^{n,2} \otimes B_A^{n,2(s-1)} \otimes B^{\bullet}.
\end{split}
\end{equation}
\end{subequations}

%-------------------------------------------------------
\subsubsection{For ambient type $D$}

Let $B^{\bullet}$ be a tensor product of KR crystals of type $D^{(1)}_{n+1}$.
We define
\begin{equation}
\label{eq:lhv_D}
\virtual{\lh} := \begin{cases} 
	\lh^D \circ \lh^D \circ \ls^D & \text{for type } B_n^{(1)},  \\ \lh^D & \text{for type } A_{2n-1}^{(2)}. 
	\end{cases}
\end{equation}
For type $B_n^{(1)}$, we define $\lhsv \colon \virtual{B}^{n,1} \otimes B^{\bullet} \to B^{\bullet}$ as the composition $\lhsv := \lhs^D \circ \lhs^D$.
Note that the first $\lhs^D$ applies to $B_D^{n,1}$ and the second to $B_D^{n+1,1}$.

Next, the operation $\virtual{\lb}$ is defined by
\begin{equation}
\label{eq:lbv_D}
\virtual{\lb} := \begin{cases} 
	\lb^{(2)D}  & \text{if } 1<r<n \text{ for type } B_n^{(1)},  \\
	\lb^D & \text{if } 1<r\leqslant n \text{ for type } A_{2n-1}^{(2)},
	\end{cases}
\end{equation}
where $\lb^{(2)D}$ is given by Definition~\ref{defn:lbs}.
For $r = n$ in type $B_n^{(1)}$, we use $\lhsv$, so we do not need this case.
For $r = n$ in type $A_{2n-1}^{(2)}$, we define $\virtual{\lb}$ by considering the natural $U_q'(\g)$-crystal 
isomorphism~\cite[Thm.~3.3]{S:2005} between $B^{n,1}_D \otimes B^{n+1,1}_D$ and the crystal of
tableaux of height $n$, where the crystal operators are given by the usual $D_n^{(1)}$ tableaux rules of~\cite{FOS:2009};
the usual $\lb^D$ map is then applied to the height $n$ tableaux.

Finally, for type $B_n^{(1)}$, we define $\virtual{\ls}$ for $1\leqslant r < n$ (resp.~for $r = n$) by the same composition of maps 
as~\eqref{eq:lsv_A_Cn} (resp.~\eqref{eq:lsv_A_generic}) with $\ls^A$ and $R^A$ replaced by $\ls^D$ and $R^D$, respectively.
For type $A_{2n-1}^{(2)}$, define $\virtual{\ls} := \ls^D$ for $1\leqslant r<n$.
For $r = n$ in type $A_{2n-1}^{(2)}$, we define $\virtual{\ls}$ as the same sequence of maps 
as~\eqref{eq:lsv_A_generic} with $\ls^A$ and $R^A$ replaced by $\ls^D$ and $R^D$, respectively.

%%%%%%%%%%%%%%%%%%%%%%%%%%%%%%%%%%%%%%%%%%
\subsection{Lifted operations on rigged configurations}
\label{subsection.lifted RC}

We define the operations $\virtual{\delta}$, $\deltasv$, $\virtual{\beta}$ and $\virtual{\gamma}$ on rigged configurations 
for the ambient types.

%------------------------------------------------
\subsubsection{For ambient type $A$}

Here we restrict ourselves to ambient type $A$.
First, for a rigged configuration with multiplicity array $L(\virtual{B}^{1,1}\ot B^{\bullet})$, we define $\virtual{\delta}$ by
\begin{equation}
\label{eq:deltav_A}
\virtual{\delta}  = \varsigma \circ \delta^A \circ \varsigma \circ \delta^A.
\end{equation}
On a rigged configuration with multiplicity array $L(\virtual{B}^{n,1} \ot B^{\bullet})$ for type $D_{n+1}^{(2)}$, we 
define $\deltasv$ by
\begin{equation}
\label{eq:deltavs_A}
\deltasv =  \delta^A \circ (\delta^A \circ \beta^A)^{n-1}.
\end{equation}

Next we define $\virtual{\beta}$ on a rigged configuration with multiplicity array $L(\virtual{B}^{r,1} \ot B^{\bullet})$ for 
$1<r<n$ in type $D_{n+1}^{(2)}$ and $1<r\leqslant n$ for all other types as
\begin{equation}
\label{eq:betav_A}
\virtual{\beta} = \begin{cases}
\varsigma \circ \beta^{A} \circ \varsigma \circ \beta^A \circ \gamma^A & \text{if } r = n \text{ for type } C_n^{(1)},
\\ \varsigma \circ \beta^{A} \circ \varsigma \circ \beta^A & \text{otherwise.}
\end{cases}
\end{equation}

Finally, we define $\virtual{\gamma}$ on a rigged configuration with multiplicity array $L(\virtual{B}^{r,s} \ot B^{\bullet})$ 
with $s>1$ as follows:
\begin{equation}
\label{eq:gammav_A}
\virtual{\gamma} = \begin{cases}
\gamma^A_{n,2s} & \text{if } r = n \text{ for type } D_{n+1}^{(2)},\\
(\gamma_{n,2}^A)^{-1} \circ \gamma_{n,2s-1}^A \circ \gamma^A_{n,2s} 
& \text{if } r = n \text{ for type } C_n^{(1)}, \\
\gamma_{2n-r,s}^A\circ\gamma^A_{r,s}  & \text{otherwise.}
\end{cases}
\end{equation}

%--------------------------------------------------
\subsubsection{For ambient type $D$}

Now we restrict ourselves to ambient type $D$.
On a rigged configuration with multiplicity array $L(\virtual{B}^{1,1}\ot \virtual{B}^{\bullet})$, we define 
$\virtual{\delta}$ by
\begin{equation}
\label{eq:deltav_D}
	\virtual{\delta} = \begin{cases} \delta^D \circ \delta^D \circ \gamma^D & \text{for type} B_n^{(1)}, \\ 
		\delta^D & \text{for type } A_{2n-1}^{(2)}. \end{cases}
\end{equation}
For type $B_n^{(1)}$, we define $\deltasv := \deltas^D \circ \deltas^D$, where the first $\deltas^D$ is for $B^{n,1}$ 
and the second is for $B^{n+1,1}$.

For type $A_{2n-1}^{(2)}$ and $1<r \leqslant n$, we define
\begin{equation}
\label{eq:betav_D}
\virtual{\beta} := \begin{cases} 
	\beta^{(2)D}  & \text{if } 1 < r < n \text{ for type } B_n^{(1)},  \\
	\beta^D & \text{if } 1 < r \leqslant n \text{ for type } A_{2n-1}^{(2)}.,
	\end{cases}
\end{equation}
where $\beta^{(2)D}$ is given by Definition~\ref{defn:bs}.
Note that for $r=n$ in type $A_{2n-1}^{(2)}$, the map $\beta^D$ is still well-defined. 
For type $B_n^{(1)}$ and $r=n$, this case does not apply as we use $\deltasv$.

Finally, we define
\begin{equation}
\label{eq:gammav_D}
	\virtual{\gamma} = \begin{cases}
	\gamma^D_{r,s} & \text{if } 1\leqslant r < n \text{ for type } A_{2n-1}^{(2)},\\
	(\gamma^D_{r,2})^{-1} \circ \gamma^D_{r,2s-1} \circ \gamma^D_{r,2s} & \text{if } 1\leqslant r < n \text{ for type } B_n^{(1)},\\
	\gamma^D_{n+1,s} \circ \gamma^D_{n,s} &\text{if $r =n$.}
\end{cases}
\end{equation}

%%%%%%%%%%%%%%%%%%%%%%%%%%%%%%%%%%%%%%%%%%%%%%
\subsection{Preparatory statements}

In order to prove our main theorem, we begin by proving some facts about the maps given in 
Sections~\ref{subsection.lifted paths} and~\ref{subsection.lifted RC}.

\begin{lemma}
\label{lemma:virtual_delta_A}
Suppose $\g$ has ambient type $A$.
We have
\[
\virtual{\delta}\bigl(\emb(\rc(L(B^{1,1}\ot B^{\bullet})))\bigr) \subset \emb\bigl(\rc(L(B^{\bullet}))\bigr).
\]
Moreover, we have $\virtual{\delta} \circ \emb = \emb \circ \delta$.
\end{lemma}

\begin{proof}
We show that $\virtual{\delta}$ is well-defined and that $\virtual{\delta} \circ \emb = \emb \circ \delta$.
We explain the most fundamental case $\g=D^{(2)}_{n+1}$ as the rest of the cases are similar.
The involution $\sigma$ is given by 
\begin{align*}
	\sigma \colon B_A^{1,1} & \longrightarrow B_A^{2n-1,1},\\
	i & \longmapsto (2n-i)^{\vee}.
\end{align*}
Here $a^{\vee}$ stands for the element of $B_A^{2n-1,1}$, which is naturally isomorphic to 
$(B_A^{1,1})^{\vee}$, without letter $a$.
Hence, the embedding $\emb \colon B^{1,1} \to \virtual{B}^{1,1} \subset B_A^{1,1}\ot B_A^{2n-1,1}$  is realized as
\begin{align*}
a & \mapsto a \ot (2n-a+1)^{\vee} \qquad\qquad (1 \leqslant a \leqslant n),
\\ \overline{a} & \mapsto (2n-a+1) \ot a^{\vee} \qquad\qquad (1 \leqslant a \leqslant n),
\\ 0 & \mapsto n \ot n^{\vee},
\\ \emptyset & \mapsto 2n \ot (2n)^{\vee}.
\end{align*}
By~\cite{SS:X=M}, the composition $\varsigma \circ \delta^A \circ \varsigma$ is similar to $\delta^A$ except 
we start at the rightmost rigged partition $(\nu, J)^{(2n-1)}$ and proceed to the left.
Therefore, doing $\virtual{\delta}$ means that we apply the usual box-removing procedure $\delta^A$ from 
the leftmost rigged partition moving right first and then perform the same procedure from the rightmost 
one moving left after.

Let us consider $2n\ot (2n)^{\vee} \in B_A^{1,1} \ot B_A^{2n-1,1}$
as an example and see what happens on the rigged configuration side. The fact that
the left factor is $2n$ means that in the first lap, a box is removed from $\nu^{(a)}$
for all $a = 1, \dotsc, 2n-1$. If the length of the string $\ell^{(2n-1)}$ 
from which a box is removed is larger
than 1, then a box in $\nu^{(2n-1)}$ can be removed in the second lap, which contradicts
that the second component is $(2n)^{\vee}$. Hence, we have $\ell^{(2n-1)} = 1$, which
forces that $\ell^{(a)} = 1$ for all $a$ in the first lap. This case corresponds to (P)
in~\cite[\S4.6]{OSS:2003a}. 

Let us take $n\ot n^{\vee}$ for next example. In this case, a box is removed 
from $\nu^{(a)}$ for $a = 1, \dotsc, n-1$ and stop in the first lap, and then a box is removed
from $\nu^{(2n-a)}$ for $a = 1, \dotsc, n$ and stop in the second lap. Since the box 
in $\nu^{(n)}$ removed in the second lap was not removed in the first lap, it should
have been quasi-singular before the procedure. This case corresponds to (Q)
in~\cite[\S4.6]{OSS:2003a}. 

Finally let us take $n+1 \ot n^{\vee}$ (which corresponds to $\overline{n}$ in type $D_{n+1}^{(2)}$).
In this case, a box is removed from $\nu^{(a)}$ for $a = 1, \dotsc, n$ and stop in the first lap, and then a box is removed
from $\nu^{(2n-a)}$ for $a=1,\dotsc,n$ and stop in the second lap. At $\nu^{(n)}$
there are two possibilities. The first one is that the string from which a box
is removed in the second lap coincides with the one in the first lap. This corresponds 
to case (S) in~\cite[\S4.6]{OSS:2003a}. The second one is that the string from which 
a box is removed in the second lap is strictly smaller than the one in the first lap.
Since the former string should have been quasi-singular before the procedure,
this corresponds to case (QS) in~\cite[\S4.6]{OSS:2003a}.

From~\cite[\S 3.4]{OSS03II}, we have $\virtual{\delta} \circ \emb = \emb \circ \delta$.
\end{proof}

\begin{remark}
The map $\virtual{\delta}$ was shown to be well-defined in~\cite[Thm.~7.1]{OSS03III} for types 
$C_n^{(1)}, A_{2n}^{(2)}, D_{n+1}^{(2)}$, but not $A_{2n}^{(2)\dagger}$.
Thus, we have given an alternative proof, which includes type $A_{2n}^{(2)\dagger}$.
\end{remark}

\begin{lemma}
\label{lemma:virtual_delta_D}
Suppose $\g$ has ambient type $D$.
We have
\[
\virtual{\delta}\bigl(\emb(\rc(L(B^{1,1}\ot B^{\bullet})))\bigr) \subset \emb\bigl(\rc(L(B^{\bullet}))\bigr).
\]
Moreover, we have $\virtual{\delta} \circ \emb = \emb \circ \delta$.
\end{lemma}

\begin{proof}
The proof is essentially done in~\cite[\S4]{OSS:2003a}.
We check for each $\g$ with ambient type $D$ that the operation $\delta$ 
is consistent with the one given there.

Consider $\g$ of type $A^{(2)}_{2n-1}$. In this case $\virtual{B}^{1,1} = B_D^{1,1}$.
Just by ignoring $(\virtual{\nu}, \virtual{J})^{(n+1)}$ since $(\virtual{\nu}, \virtual{J})^{(n+1)} = (\virtual{\nu}, \virtual{J})^{(n)}$ (so $\virtual{\delta} = \delta^D$ selects the same row in both rigged partitions), we obtain the algorithm in~\cite[\S4.5]{OSS:2003a}.

For $\g$ of type $B^{(1)}_n$, we have $\virtual{B}^{1,1} = B_D^{1,2}$.
As an illustration let us look at the case when $n\bar{n} \in B_D^{1,2}$ corresponding
to $0 \in B^{1,1}$ is removed and compare it with the algorithm of \cite[\S4.2]{OSS:2003a} as the other cases are similar.
Note that all the configurations $\virtual{\nu}^{(a)}$ in the ambient rigged configurations are doubled from the ones there for all $a < n$.
The first case of (Q) occurs exactly when a box is removed from a string of 
length $2\ell^{(a)}$ in $\virtual{\nu}^{(a)}$ for $a=1, \dotsc, n-1$ and stopped in the first application of $\delta^D$
and a box is removed from a string of length $2\ell^{(a)}-1$ for $a=1, \dotsc, n-1$ and also from a string of length $2\ell^{(n-1)}-1$ and stopped in the second application.
The second case of (Q) is the same as the previous one except that in the second
application of $\delta^D$ the box is removed from a string of length not less than $2\ell^{(n-1)}$.
This box was not removed in the first application of $\delta^D$, since it was quasi-singular.

From~\cite[Prop.~3.9]{OSS03II}, we have $\virtual{\delta} \circ \emb = \emb \circ \delta$.
\end{proof}

\begin{proposition}
\label{prop:virtualization_B_to_D}
Let $\g$ be of type $B_n^{(1)}$.
Consider an element $(s_1, \dotsc, s_n) \in B^{n,1}$ given by the $\pm$-vector from~\cite{KN:1994}.
Denote the elements in $B^{r,1}$, for $r = n, n+1$, of type $D_{n+1}^{(1)}$ by $\pm$-vectors from~\cite{KN:1994} as well.
Define $\emb \colon B^{n,1} \to \virtual{B}^{n,1} = B^{n,1} \otimes B^{n+1,1}$ by
\[
\emb(s_1, \dotsc, s_n) = (s_1, \dotsc, s_n, s_{n+1}^-) \otimes (s_1, \dotsc, s_n, s_{n+1}^+),
\]
where $s_{n+1}^{\pm}$ is such that
\[
s_1 \cdots s_n s_{n+1}^{\pm} = \pm 1.
\]
Then $\emb$ is the virtualization map and the image is characterized by $\sigma(b) = R(b)$.
\end{proposition}

\begin{proof}
This is straightforward from the definition of the (virtual) crystal operators.
The uniqueness comes from the fact that tensor products of KR crystals are generated by the (unique) maximal vector.
The characterization of the images comes from the fact that the maximal vector must map to the maximal vector, 
which has the desired property.
\end{proof}

\begin{proposition} \label{prop:emb_rc}
The operations $\virtual{\delta}$, $\virtual{\beta}$ and $\virtual{\gamma}$ for ambient types $A$ and $D$
as well as $\deltasv$ for types $D^{(2)}_{n+1}$ and $B_n^{(1)}$ have the following properties:
\begin{itemize}
\item[(1)] $\virtual{\delta}\bigl(\emb(\rc(L(B^{1,1}\ot B^{\bullet})))\bigr) \subset \emb\bigl(\rc(L(B^{\bullet}))\bigr)$,
\item[(1')] $\deltasv\bigl(\emb(\rc(L(B^{n,1}\ot B^{\bullet})))\bigr) \subset \emb\bigl(\rc(L(B^{\bullet}))\bigr)$,
\item[(2)] $\virtual{\beta}\bigl(\emb(\rc(L(B^{r,1}\ot B^{\bullet})))\bigr) \subset \emb\bigl(\rc(L(B^{1,1}\ot B^{r-1,1}\ot B^{\bullet}))\bigr)$,
\item[(3)] $\virtual{\gamma}\bigl(\emb(\rc(L(B^{r,s}\ot B^{\bullet})))\bigr) \subset \emb\bigl(\rc(L(B^{r,1}\ot B^{r,s-1}\ot B^{\bullet}))\bigr)$.
\end{itemize}
\end{proposition}

\begin{proof}
The proof of~(1) is given by Lemma~\ref{lemma:virtual_delta_A} (resp.\ Lemma~\ref{lemma:virtual_delta_D}) 
for ambient type $A$ (resp.\ ambient type $D$).
For~(1') in type $D_{n+1}^{(2)}$, this follows from~\cite[Lemma~3.9]{Scrimshaw17} and the fact
that there exists a virtualization map (see also~\cite[Prop.~8.3]{Scrimshaw17}).
For~(1') in type $B_n^{(1)}$, this is given by~\cite[Thm.~6.1]{Scrimshaw17}.
For~(2), this is a straightforward computation.
For~(3), this is the identity map on rigged configurations and a straightforward computation of the changes 
in vacancy numbers shows that $\virtual{\gamma}$ is well-defined.
\end{proof}

\begin{remark}
For $r < n$ in type $A_{2n-1}^{(2)}$, we note that Proposition~\ref{prop:emb_rc} was shown 
in~\cite[Thm.~6.2]{schilling.scrimshaw.2015}.
\end{remark}

\begin{proposition} \label{prop:emb path}
The operations $\virtual{\lh}$, $\virtual{\lb}$ and $\virtual{\ls}$ for ambient types $A$ and $D$ as well
as $\lhsv$ for types $D^{(2)}_{n+1}$ and $B_n^{(1)}$ have the following properties:
\begin{itemize}
\item[(1)] $\virtual{\lh}\bigl(\emb(\p(B^{1,1}\ot B^{\bullet}))\bigr) \subset \emb\bigl(\p(B^{\bullet})\bigr)$,
\item[(1')] $\lhsv\bigl(\emb(\p(B^{n,1}\ot B^{\bullet}))\bigr) \subset \emb\bigl(\p(B^{\bullet})\bigr)$,
\item[(2)] $\virtual{\lb}\bigl(\emb(\p(B^{r,1}\ot B^{\bullet}))\bigr) \subset \emb\bigl(\p(B^{1,1}\ot B^{r-1,1}\ot B^{\bullet})\bigr)$,
\item[(3)] $\virtual{\ls}\bigl(\emb(\p(B^{r,s}\ot B^{\bullet}))\bigr) \subset \emb\bigl(\p(B^{r,1}\ot B^{r,s-1}\ot B^{\bullet})\bigr)$.
\end{itemize}
\end{proposition}

\begin{proof}
The proofs of~(1) and~(1') follow immediately from the fact $\emb$ is defined component-wise.
Next, we show~(2) by induction on $r$, where the base case $r = 1$ is trivial.
Recall that $\Phi^X$ intertwines $\lh^X$, $\lb^X$, $\ls^X$ and $R^X$ with $\delta^X$, $\beta^X$, $\gamma^X$ and 
$\id^X$ respectively for ambient type $A$~\cite{KSS:2002} and ambient type $D$~\cite{OSSS16},
as well as $\lb^{(s)D}$ with $\beta^{(s)D}$ by Proposition~\ref{prop:general_lb_beta}.
Moreover, recall that $\Phi^X$ intertwines $\sigma$ and $\varsigma$ by Proposition~\ref{prop:involution}.
Hence, the bijection $\Phi^X$ intertwines $\virtual{\lb}$ with $\virtual{\beta}$.
From Proposition~\ref{prop:emb_rc}, the map $\virtual{\beta}$ is well-defined and, thus, so is the map $\virtual{\lb}$.
Since $\virtual{\lb}$ only affects the leftmost factor(s) and is a composition of strict $U_q(\g_0)$-crystal embeddings, 
it is sufficient to consider $\virtual{p} \in \emb\bigl(\p(B^{r,1})\bigr)$.
We note that $\emb(B^{r,1})$ is multiplicity free as a $U_q(\g_0)$-crystal, and therefore 
$(\nu^X, J^X) = \Phi(\virtual{p}) \in \emb\bigl(\rc(L(B^{r,1}))\bigr)$.
Next, we note that $\virtual{\beta}(\nu^X, J^X) \in \emb\bigl(\rc(L(B^{1,1} \otimes B^{r-1,1}))\bigr)$ by 
Proposition~\ref{prop:emb_rc}.
Therefore, we have $\virtual{\lb}(\virtual{p}) \in \emb\bigl(\p(B^{1,1} \otimes B^{r-1,1})\bigr)$ by induction.

The proof of~(3) is similar to~(2) by using induction on $s$.
Additionally note that $(\ls^A)^{-1}$ in~\eqref{eq:lsv_A_Cn} can indeed be applied because of the previous 
two applications of $\ls^A$ and the corresponding rigged configuration is never changed.
\end{proof}

\begin{lemma}
\label{lemma:virtual_lx}
We have
\[
\virtual{\lx} \circ \emb = \emb \circ \lx
\]
for $\lx = \lh, \lhs, \lb, \ls$.
\end{lemma}

\begin{proof}
Since $\emb$ is applied component-wise and $B^{r,s}$ is multiplicity free as a $U_q(\g_0)$-crystal, we have $\virtual{\lx} \circ \emb = \emb \circ \lx$.
\end{proof}

\begin{lemma}
\label{lemma:virtual_xi}
We have
\[
\virtual{\xi} \circ \emb = \emb \circ \xi
\]
for $\xi = \delta, \deltas, \beta, \gamma$.
\end{lemma}

\begin{proof}
For $\xi = \delta$, the claim was shown in Lemma~\ref{lemma:virtual_delta_A} (resp.~Lemma~\ref{lemma:virtual_delta_D}) 
for ambient type $A$ (resp.~ambient type $D$).
For $\xi = \beta$, this is a straightforward computation on the change in vacancy numbers and that $\beta$ and 
$\virtual{\beta}$ do not change the vacancy numbers.
For $\xi = \gamma$, this is a straightforward computation on the vacancy numbers.
For $\xi = \deltas$, the statement for type $D_{n+1}^{(2)}$ (resp.~type $B_n^{(1)}$) was shown 
in~\cite[Prop.~8.3]{Scrimshaw17} (resp.~\cite[Thm.~6.1]{Scrimshaw17}). 
\end{proof}

%%%%%%%%%%%%%%%%%%%%%%%%%%%%%%%%%%%%%%%%%%%%%%
\subsection{Main theorem}

We now give our main results.
We will use diagrams of the following kind:
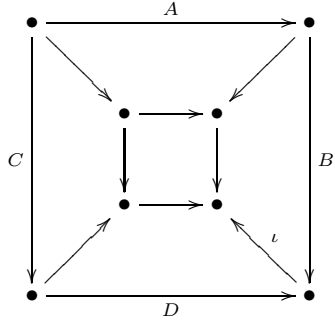
\begin{equation*}
\xymatrix{
 {\bullet} \ar[rrr]^{A} \ar[ddd]_{C} \ar[dr] & & &
        {\bullet} \ar[ddd]^{B} \ar[dl] \\
 & {\bullet} \ar[r] \ar[d] & {\bullet} \ar[d] & \\
 & {\bullet} \ar[r]   & {\bullet}  & \\
 {\bullet} \ar[rrr]_{D} \ar[ur] & & & {\bullet} \ar[ul]_{\iota}
}
\end{equation*}
We regard this diagram as a cube with front face given by the large square.
Suppose that the square diagrams given by the faces of the cube except for the front face commute and $\iota$ 
is the injective map. Then the front face also commutes since we have
\[
\iota \circ B \circ A = \iota \circ D \circ C
\]
by diagram chasing~\cite[Lemma~5.3]{KSS:2002}.

\begin{theorem}
\label{theorem.main}
Let $B$ be a tensor product of KR crystals with multiplicity array $L$.
Then there exists a unique family of bijections $\Phi \colon \p(B,\la)\rightarrow\rc(L,\la)$ such that
\[
\Phi^X \circ \emb = \emb \circ \Phi,
\]
and the empty path maps to the empty rigged configuration.
It satisfies the following commutative diagrams.

\begin{enumerate}
\item[(1)]
Suppose $B = B^{1,1} \otimes B'$. Let $\lh(B) = B'$ with multiplicity array $\lh(L)$. Then the diagram
\[
\xymatrix{
\p(B,\lambda) \ar[r]^{\Phi} \ar[d]_{\lh} & \rc(L,\lambda) \ar[d]^{\delta} \\
\bigcup_{\mu}\p(\lh(B),\mu) \ar[r]_{\Phi} & \bigcup_{\mu}\rc(\lh(L),\mu)
}
\]
commutes.

\item[(1')]
For type $D_n^{(1)}$, $B_n^{(1)}$ and $D_{n+1}^{(2)}$ when the left-most factor is $B^{n,1}$ (or $B^{n-1,1}$ in type $D_n^{(1)}$), the diagram
\[
\xymatrix{
\p(B,\lambda) \ar[r]^{\Phi} \ar[d]_{\lhs} & \rc(L,\lambda) \ar[d]^{\deltas} \\
\bigcup_{\mu}\p(\lhs(B),\mu) \ar[r]_{\Phi} & \bigcup_{\mu}\rc(\lhs(L),\mu)
}
\]
commutes.

\item[(2)]
Suppose $B=B^{r,1}\otimes B'$ with $2\leqslant r\leqslant n$.
Let $\lb (B)=B^{1,1}\otimes B^{r-1,1}\otimes B'$ with multiplicity array $\lb (L)$.
Then the diagram
\[
\xymatrix{
\p(B,\lambda) \ar[r]^{\Phi} \ar[d]_{\lb} & \rc(L,\lambda) \ar[d]^{\beta} \\
\p(\lb(B),\lambda) \ar[r]_{\Phi} & \rc(\lb(L),\lambda)
}
\]
commutes.

\item[(3)]
Suppose $B=B^{r,s}\otimes B'$ with $s\geqslant 2$.
Let $\ls(B)=B^{r,1}\otimes B^{r,s-1}\otimes B'$ with multiplicity array $\ls(L)$.
Then the diagram
\[
\xymatrix{
\p(B,\lambda) \ar[r]^{\Phi} \ar[d]_{\ls} & \rc(L,\lambda) \ar[d]^{\gamma}\\
\p(\ls(B),\lambda) \ar[r]_{\Phi} & \rc(\ls(L),\lambda)
}
\]
commutes.
\end{enumerate}
\end{theorem}

\begin{proof}
For type $A_n^{(1)}$ (resp.~$D_n^{(1)}$), the claim was shown in~\cite{KSS:2002} 
(resp.~\cite{OSSS16}), where we consider $\emb$ as the identity map.
The other cases are shown by embedding both $\p(B)$ and $\rc(L)$ into those of type $A_{2n-1}^{(1)}$ (resp.~$D_{n+1}^{(1)}$) for ambient type $A$ (resp.~$D$).

Consider the following diagram:
%\[
%\xymatrix{
%\p(B)\; \ar@{^{(}->}[r]^{\emb} \ar@{.>}[d] & \vp(\virtual{B})
%\ar[r]^{\Phi^X} \ar[d]_{\virtual{\lx}} & 
%\vrc(\virtual{L}) \ar[d]^{\virtual{\xi}}  \ar@{<-^{)}}[r]^{\emb} & 
%\;\rc(L) \ar@{.>}[d] \\
%\p(B')\; \ar@{^{(}->}[r]_{\emb} & \vp(\virtual{B}') \ar[r]_{\Phi^X} & 
%\vrc(\virtual{L}')) \ar@{<-^{)}}[r]_{\emb} & \;\rc(L')
%}
%\]
\begin{equation} \label{proof of main th}
\begin{gathered}  % Little trick to center the equation number with xymatrix
\xymatrixcolsep{4em}
%\xymatrixrowsep{3em}
\xymatrix{
 {\p(B)}\; \ar@{.>}[rrr]^{\Phi} \ar[ddd]_{\lx} \ar@{^{(}->}[dr]_{\emb} & & &
        {\rc(L)} \ar[ddd]^{\xi} \ar@{^{(}->}[dl]^{\emb} \\
 & {\vp(\virtual{B})} \ar[r]^{\Phi^X} \ar[d]_{\virtual{\lx}} & {\vrc(\virtual{L})} \ar[d]^{\virtual{\xi}} & \\
 & {\vp(\virtual{B}')} \ar[r]_{\Phi^X}   & {\vrc(\virtual{L}')}  & \\
 {\p(B')} \ar[rrr]_{\Phi} \ar@{^{(}->}[ur]^{\emb} & & & {\;\rc(L')} \ar@{^{(}->}[ul]_{\emb}
}
\end{gathered}
\end{equation}
where $\lx = \lh, \lhs, \lb, \ls$ and $\xi = \delta, \deltas, \beta, \gamma$ respectively,
$B' = \lx(B)$ and $L' = \lx(L)$.
Thanks to Proposition~\ref{prop:emb path}, the left face is well-defined, and it commutes by Lemma~\ref{lemma:virtual_lx}.
The right face is also well-defined due to Proposition~\ref{prop:emb_rc} and commutes by Lemma~\ref{lemma:virtual_xi}.
The back face commutes since $\virtual{\lx}$ and $\virtual{\xi}$ are compositions of maps, as defined in
Section~\ref{subsection.lifted paths}, intertwining with $\Phi^X$ by~\cite{KSS:2002, OSSS16};
Proposition~\ref{prop:general_lb_beta}; and Proposition~\ref{prop:involution}.

We first show that there exists an injective map $\Phi$ from $\p(B)$ to $\rc(L)$. We use
an induction on $B$ such that the application of any $\lx$ on $B$
decreases its order. Suppose in the bottom of~\eqref{proof of main th} that there exists
an injective map from $\p(B')$ to $\rc(L')$ and take $p\in\p(B)$ such that 
$\lx(p)\in\p(B')$. Let $(\nu',J')=\Phi\bigl(\lx(p)\bigr)$. 
Then from the commutativity of the left, right and back faces of diagram~\eqref{proof of main th} and the bijectivity of $\Phi^X$, there must exist
$(\virtual{\nu},\virtual{J})$ such that $\virtual{\xi}(\virtual{\nu},\virtual{J}) = \emb(\nu',J')$.
However, from the explicit algorithm of $\virtual{\xi}$, the rigged configuration
$(\virtual{\nu},\virtual{J})$ belongs to $\emb\bigl(\rc(L)\bigr)$.
Hence, to $p\in\p(B)$ one can associate  $(\nu,J)\in\rc(L)$ such that $\emb(\nu,J) = (\virtual{\nu},\virtual{J})$.
One can thus define  $\Phi(p) = (\nu,J)$.

To show that $\Phi$ is a bijection, we make \eqref{proof of main th} left and right reversed
and replace $\Phi^X$ with $(\Phi^X)^{-1}$.
Moreover, the top and bottom faces commute by the definition of $\Phi$.
Hence, we have Theorem~\ref{theorem.main}.
\end{proof}

\begin{remark}
Proposition~\ref{prop:involution}, Proposition~\ref{prop:virtualization_B_to_D} and~\cite{KSS:2002, OSSS16} imply Theorem~\ref{theorem.main} for types 
$D_{n+1}^{(2)}$ and $A_{2n-1}^{(2)}$ since $\emb\bigl(\p(B)\bigr)$ (resp.~$\emb\bigl(\rc(L)\bigr)$) is characterized by being 
invariant under $\sigma$ (resp.~$\varsigma$)~\cite{OSS03II}.
\end{remark}

%%%%%%%%%%%%%%%%%%%%%%%%%%%%%%%%%%%%%%%%%%%%%%%%%%%
\section{Properties of $\Phi$ and $X=M$}
\label{sec:properties}
%%%%%%%%%%%%%%%%%%%%%%%%%%%%%%%%%%%%%%%%%%%%%%%%%%%

In this section, we show that the intrinsic energy function and cocharge are related by the bijection $\Phi$,
which results in a (combinatorial) proof of the $X=M$ conjecture~\cite{HKOTY99, HKOTT02} for all nonexceptional types.

\begin{proposition}
\label{prop:virtual_box}
Let $\g$ be of nonexceptional type with \emph{any} embedding $\emb$ from $\virtual{\g}$.

\begin{enumerate}
\item[(1)]
For $\Box \in \{ \lusz, \vee, \hwstar\}$, the diagram
\[
\xymatrix{
\p(B,\lambda) \ar[r]^{\emb} \ar[d]_{\Box} & \vp(\virtual{B},\virtual{\lambda}) \ar[d]^{\Box}\\
\p(B^{\Box},\lambda) \ar[r]_{\emb} & \vp(\virtual{B}^{\Box},\virtual{\lambda})
}
\]
commutes.

\item[(2)]
The diagram
\[
\xymatrix{
\rc(L,\lambda) \ar[r]^{\emb} \ar[d]_{\theta} & \vrc(\virtual{L},\virtual{\lambda}) \ar[d]^{\theta}\\
\rc(L,\lambda) \ar[r]_{\emb} & \vrc(\virtual{L},\virtual{\lambda})
}
\]
commutes.
\end{enumerate}
\end{proposition}

\begin{proof}
We first show~(1) for $\Box = \lusz$.
Let $\overline{u}_{\lambda}$ denote the lowest weight vector in $B(\lambda)$.
We have
\[
\bigl( \emb(u_{\lambda}) \bigr)^{\lusz}  = u_{\virtual{\lambda}}^{\lusz} = \overline{u}_{\virtual{\lambda}} = \emb(\overline{u}_{\lambda}) = \emb(u_{\lambda}^{\lusz}).
\]
We also have
\begin{align*}
\bigl(\emb(e_i b) \bigr)^{\lusz} & = \bigl( \virtual{e}_i \emb(b) \bigr)^{\lusz} = \left(\prod_{j \in \phi^{-1}(i)} (e^X_j)^{\gamma_i} \emb(b) \right)^{\lusz}
\\ & = \prod_{j \in \phi^{-1}(i)} (f^X_{\tau(j)})^{\gamma_i} \bigl(\emb(b)\bigr)^{\lusz} = \virtual{f}_i \bigl(\emb(b)\bigr)^{\lusz},
%\\ \bigl(\emb(f_i b) \bigr)^{\lusz} & = \bigl( \virtual{f}_i \emb(b) \bigr)^{\lusz} = \left(\prod_{j \in \phi^{-1}(i)} (f^X_j)^{\gamma_i} \emb(b) \right)^{\lusz}
%\\ & = \prod_{j \in \phi^{-1}(i)} (e^X_{\tau(j)})^{\gamma_i} \bigl(\emb(b)\bigr)^{\lusz} = \virtual{e}_i \bigl(\emb(b)\bigr)^{\lusz},
\end{align*}
since the orbits of $\tau$ are contained in the orbits of $\phi$, and similarly interchanging $e_i$ and $f_i$.
Recall that $u_{\lambda}$ generates $B(\lambda)$, thus the claim follows from~\eqref{eq:box_factors} and Proposition~\ref{prop:virtual_tensor_category}.
The proof for $\Box = \vee$ is similar.
The proof for $\Box = \hwstar$ follows from~\cite[Thm.~5.1]{Okado.2013} and the $\Box = \lusz$ case.

For~(2), the claim follows from~\eqref{eq:virtual_vacancy}.
\end{proof}

\begin{proposition}
\label{prop:diamond_theta}
Let $\g$ be of nonexceptional type.
The diagram
\[
\xymatrix{
\p(B,\lambda) \ar[r]^{\Phi} \ar[d]_{\hwstar} & \rc(L,\lambda) \ar[d]^{\theta}\\
\p(B,\lambda) \ar[r]_{\Phi} & \rc(L,\lambda)
}
\]
commutes.
\end{proposition}

\begin{proof}
Consider the cube:
\begin{equation}
\label{eq:theta_cube}
\begin{gathered}  % Little trick to center the equation number with xymatrix
\xymatrix{
\p(B,\lambda) \ar[rrr]^{\Phi} \ar[dr]^{\emb} \ar[ddd]_{\hwstar} & & & \rc(L,\lambda) \ar[ddd]^{\theta} \ar[dl]_{\emb} \\
& \vp(\virtual{B}, \virtual{\lambda}) \ar[r]^{\Phi^X} \ar[d]_{\hwstar} & \vrc(\virtual{L}, \virtual{\lambda}) \ar[d]^{\theta} & \\
& \vp(\virtual{B}, \virtual{\lambda}) \ar[r]_{\Phi^X} & \vrc(\virtual{L}, \virtual{\lambda}) & \\
\p(B,\lambda) \ar[rrr]_{\Phi} \ar[ur]_{\emb} & & & \rc(L,\lambda) \ar[ul]^{\emb}
}
\end{gathered}
\end{equation}
The left face commutes by Proposition~\ref{prop:virtual_box}(1).
The right face commutes by Proposition~\ref{prop:virtual_box}(2).
The top and bottom faces commute by Theorem~\ref{theorem.main}.
The back face commutes when $\virtual{\g}$ is of type $A_n^{(1)}$ by~\cite{KSS:2002} and of type $D_n^{(1)}$ by~\cite[Prop.~4.1(7)]{OSSS16}.
Hence, the front face commutes and the claim follows.
\end{proof}

\begin{proposition}
\label{prop:commuting_paths}
When there are at least two KR crystals in the tensor product $B$, the left operation $\mathrm{lx}$ commutes with the right one $\mathrm{ry}$ for any pair of $(\mathrm{x,y})$, where $\mathrm{x,y=h,h_{\mathrm{sp}},b,s}$, as long as they are well-defined.
\end{proposition}

\begin{proof}
This follows from Proposition~\ref{prop:virtual_box}, Theorem~\ref{theorem.main} and the corresponding statements in types $A_n^{(1)}$~\cite{KSS:2002} and $D_n^{(1)}$~\cite[Prop~3.12]{OSSS16}.
\end{proof}

\begin{proposition}
\label{prop:commuting_rc}
Let $\g$ be of nonexceptional type.
The operation $\xi$ commutes with the operation $\widetilde{\zeta}$ for all $\xi, \zeta = \delta, \deltas, \beta, \gamma$ as long as they are well-defined.
%We have the following relations on rigged configurations in $\rc(L)$:
%\begin{enumerate}
%\item $[\delta,\widetilde{\delta}] = [\delta,\deltast] = [\delta,\widetilde{\beta}] = [\delta,\widetilde{\gamma}] = 0$;
%\item $[\deltas,\widetilde{\delta}] = [\deltas,\deltast] = [\deltas,\widetilde{\beta}] = [\deltas,\widetilde{\gamma}] = 0$;
%\item $[\beta,\widetilde{\delta}] = [\beta,\deltast] = [\beta,\widetilde{\beta}] = [\beta,\widetilde{\gamma}] = 0$;
%\item $[\gamma,\widetilde{\delta}] = [\gamma,\deltast] = [\gamma,\widetilde{\beta}] = [\gamma,\widetilde{\gamma}] = 0$.
%\end{enumerate}
\end{proposition}

\begin{proof}
This follows from Proposition~\ref{prop:virtual_box}, Theorem~\ref{theorem.main} and the corresponding statements in types $A_n^{(1)}$~\cite{KSS:2002} and $D_n^{(1)}$~\cite[Prop~3.12, Prop~3.16]{OSSS16}.
\end{proof}

\begin{proposition}
\label{prop:right_version}
Let $\g$ be of nonexceptional type.
Then under $\Phi$, the maps $\mathrm{rx}$ correspond to $\widetilde{\xi}$, where $\mathrm{x} = \mathrm{h}, \mathrm{h}_{\mathrm{sp}}, \mathrm{b, s}$ and $\xi = \delta, \deltas, \beta, \gamma$ respectively.
%\begin{enumerate}
%\item[(1)]
%The diagram
%\[
%\xymatrix{
%\p(B,\lambda) \ar[r]^{\Phi} \ar[d]_{\rh} & \rc(L,\lambda) \ar[d]^{\widetilde{\delta}} \\
%\bigcup_{\mu}\p(\rh(B),\mu) \ar[r]_{\Phi} & \bigcup_{\mu}\rc(\rh(L),\mu)
%}
%\]
%commutes.
%
%\item[(2)]
%The diagram
%\[
%\xymatrix{
%\p(B,\lambda) \ar[r]^{\Phi} \ar[d]_{\rh_{\mathrm{sp}}} & \rc(L,\lambda) \ar[d]^{\deltast} \\
%\bigcup_{\mu}\p(\rh_{\mathrm{sp}}(B),\mu) \ar[r]_{\Phi} & \bigcup_{\mu}\rc(\rh_{\mathrm{sp}}(L),\mu)
%}
%\]
%commutes.
%
%\item[(3)]
%Suppose $B = B' \otimes B^{r,1}$ with $2\leqslant r\leqslant n$.
%Let $\rb(B) = B'\otimes B^{r-1,1}\otimes B^{1,1}$ with multiplicity array $\rb(L)$.
%Then the diagram
%\[
%\xymatrix{
%\p(B,\lambda) \ar[r]^{\Phi} \ar[d]_{\rb} & \rc(L,\lambda) \ar[d]^{\widetilde{\beta}} \\
%\p(\rb(B),\lambda) \ar[r]_{\Phi} & \rc(\rb(L),\lambda)
%}
%\]
%commutes.
%
%\item[(4)]
%Suppose $B = B' \otimes B^{r,s}$ with $s\geqslant 2$.
%Let $\rs(B) = B' \otimes B^{r,s-1} \otimes B^{r,1}$ with multiplicity array $\rs(L)$.
%Then the diagram
%\[
%\xymatrix{
%\p(B,\lambda) \ar[r]^{\Phi} \ar[d]_{\rs} & \rc(L,\lambda) \ar[d]^{\widetilde{\gamma}}\\
%\p(\rs(B),\lambda) \ar[r]_{\Phi} & \rc(\rs(L),\lambda)
%}
%\]
%commutes.
%\end{enumerate}
\end{proposition}

\begin{proof}
This follows from Theorem~\ref{theorem.main} and Proposition~\ref{prop:diamond_theta} by using a cube similar to~\eqref{eq:theta_cube}.
\end{proof}

\begin{proposition}
\label{prop:R_identity}
Let $\g$ be of nonexceptional type.
Then the diagram
\[
\xymatrix{
\p(B,\lambda) \ar[r]^{\Phi} \ar[d]_{R} & \rc(L,\lambda) \ar[d]^{\id}\\
\p(B',\lambda) \ar[r]_{\Phi} & \rc(L,\lambda)
}
\]
commutes.
\end{proposition}

\begin{proof}
Consider the cube:
\begin{equation}
\label{eq:phi_R_cubed}
\begin{gathered}  % Little trick to center the equation number with xymatrix
\xymatrix{
\p(B,\lambda) \ar[rrr]^{\Phi} \ar[dr]^{\emb} \ar[ddd]_{R} & & & \rc(L,\lambda) \ar[ddd]^{\id} \ar[dl]_{\emb} \\
& \vp(\virtual{B}, \virtual{\lambda}) \ar[r]^{\Phi^X} \ar[d]_{\virtual{R}} & \vrc(\virtual{L}, \virtual{\lambda}) \ar[d]^{\id} \\
& \vp(\virtual{B}', \virtual{\lambda}) \ar[r]_{\Phi^X} & \vrc(\virtual{L}, \virtual{\lambda}) & \\
\p(B',\lambda) \ar[rrr]_{\Phi} \ar[ur]_{\emb} & & & \rc(L,\lambda) \ar[ul]^{\emb}
}
\end{gathered}
\end{equation}
where $\virtual{R}$ is the corresponding combination of combinatorial $R^X$-matrices to obtain 
$\virtual{R} \colon \virtual{B} \to \virtual{B}'$.
The right face trivially commutes.
The top and bottom faces commute by Theorem~\ref{theorem.main}.
The back face commutes in type $A_{2n-1}^{(1)}$ by~\cite{KSS:2002} and $D_{n+1}^{(1)}$ by~\cite[Thm.~5.11]{OSSS16}.
To show the left face commutes, we first let $u$, $u'$, $\virtual{u}$ and $\virtual{u}'$ denote the maximal elements of 
$B$, $B'$, $\virtual{B}$ and $\virtual{B}'$ respectively.
Note that $\emb(u) = \virtual{u}$ and $\emb(u') = \virtual{u}'$.
By~\eqref{eq:ambient_emb}, 
Proposition~\ref{prop:virtual_tensor_category},~\cite[Thm.~5.1]{Okado.2013} and that the combinatorial 
$R$-matrix (resp.~$R^X$-matrix) is defined by $u \mapsto u'$ (resp.~$\virtual{u} \mapsto \virtual{u}'$), the left face commutes.
Therefore, the front face commutes as desired.
\end{proof}

Let $\cc \colon \rc(L) \to \ZZ$ denote the \defn{cocharge} statistic on rigged configurations and is defined, 
following~\cite[(3.3)]{S06II}, by
\[
	\cc(\nu, J) = \frac{1}{2}\sum_{a \in I_0} \sum_{i \in \ZZ_{>0}} t_a^{\vee} m_i^{(a)} \left( \sum_{j \in \ZZ_{>0}} L_j^{(a)} 
	- p_i^{(a)} \right) + \sum_{x \in J_i^{(a)}} t_a^{\vee} x,
\]
where $t_a^{\vee} = \max(c_i^{\vee} / c_i, c_0)$.\footnote{If we consider the $t_a^{\vee}$ as given~\cite{S06II}, we 
additionally require $t_a^{\vee} = 1$ for all $a$ when $\g$ is of type $A_{2n}^{(2)\dagger}$.}

Next, we recall the definition of the \defn{intrinsic energy} statistic $D \colon \p(B) \to \ZZ$ on paths~\cite[(2.13)]{OSS03III} 
(alternatively~\cite[(3.8)]{HKOTT02}, but the tensor factors are reversed from our tensor product convention).
For KR crystals $B^{r,s}$ and $B^{r',s'}$, the \defn{local energy function} $H \colon B^{r,s} \otimes B^{r',s'} \to \ZZ$ is defined by 
\begin{equation}
\label{eq:local_energy}
H\bigl( e_i(b \otimes b') \bigr) = H(b \otimes b') + \begin{cases}
-1 & \text{if } i = 0 \text{ and (LL)}, \\
1 & \text{if } i = 0 \text{ and (RR)}, \\
0 & \text{otherwise,}
\end{cases}
\end{equation}
where, for $\widetilde{b}' \otimes \widetilde{b} = R(b \otimes b')$, we have the conditions:
\begin{itemize}
\item[(LL)] $e_0(b \otimes b') = e_0 b \otimes b' \text{ and } e_0(\widetilde{b}' \otimes \widetilde{b}) = 
e_0 \widetilde{b}' \otimes \widetilde{b}$;
\item[(RR)] $e_0(b \otimes b') = b \otimes e_0 b' \text{ and } e_0(\widetilde{b}' \otimes \widetilde{b}) = 
\widetilde{b}' \otimes e_0 \widetilde{b}$.
\end{itemize}
It is known that $H$ is uniquely defined up to an additive constant~\cite{KKMMNN91}. We normalize $H$ by the 
condition $H( u_{s\kappa_r \clfw_r} \otimes u_{s'\kappa_{r'} \clfw_{r'}} ) = 0$.

Next consider $B^{r,s}$, and define $D_{B^{r,s}} \colon B^{r,s} \to \ZZ$ as follows.
Consider $b \in B(\lambda) \subseteq B^{r,s}$ and the classical weights $\lambda$ and $\mu = s \clfw_r$ as partitions.
Then, we define
\[
D_{B^{r,s}}(b) = \begin{cases}
\frac{1}{2} \bigl( \lvert \mu \rvert - \lvert \lambda \rvert \bigr) & \text{if } 
\g = D_n^{(1)}, B_n^{(1)}, A_{2n-1}^{(2)}, C_n^{(1)}, A_{2n}^{(2)\dagger}, \\
\lvert \mu \rvert - \lvert \lambda \rvert& \text{otherwise}.
\end{cases}
\]
%see also~\cite{HKOTT02}
Let $B = \bigotimes_{i=1}^N B^{r_i,s_i}$. We define the \defn{energy}~\cite{HKOTY99} $D \colon B \to \ZZ$ by
\[
	D = \sum_{1 \leqslant i < j \leqslant N} H_i R_{i+1} R_{i+2} \cdots R_{j-1} 
	+ \sum_{j=1}^N D_{B^{r_j,s_j}} R_1 R_2 \cdots R_{j-1},
\]
where $R_i$ and $H_i$ are the combinatorial $R$-matrix and local energy function, respectively, acting on the $i$-th and 
$(i+1)$-th factors and $D_{B^{r_j,s_j}}$ acts on the rightmost factor. Note that $D$ is constant on classical components 
since $H$ is and $R$ is a $U_q'(\g)$-crystal isomorphism.

\begin{theorem}
Let $\g$ be of nonexceptional type. Then
\[
D = \cc \circ \theta \circ \Phi.
\]
\end{theorem}

\begin{proof}
We first note that the claim was shown for type $A_n^{(1)}$ in~\cite{KSS:2002} and $D_n^{(1)}$ in~\cite{OSSS16}.
For the remaining nonexceptional types $\g$, it is sufficient to show that 
\begin{subequations}
\begin{align}
\label{eq:virtual_cocharge}
\cc(\virtual{\nu}, \virtual{J}) & = \gamma_0 \cc(\nu, J),
\\
\label{eq:virtual_energy}
D\bigl( \emb(b) \bigr) & = \gamma_0 D(b),
\end{align}
\end{subequations}
as the claim follows from the corresponding $\virtual{\g}$ case.
We have~\eqref{eq:virtual_cocharge} from~\cite[Thm.~4.2]{OSS03II}.
To show~\eqref{eq:virtual_energy}, it is sufficient to show that
\begin{equation}
\label{eq:virtual_local_energy}
H\bigl(\emb(b_1 \otimes b_2)\bigr) = \gamma_0 H(b_1 \otimes b_2),
\end{equation}
where $H \colon B^{r,s} \otimes B^{r',s'} \to \ZZ$ is the local energy function~\cite{KKMMNN91}.
From~\cite[Thm.~5.1]{Okado.2013} and the fact that the left face of the cube~\eqref{eq:phi_R_cubed} commutes, we have~\eqref{eq:virtual_local_energy}.
\end{proof}

\begin{corollary}
Let $\g$ be of nonexceptional type.
Then the $X = M$ conjecture of~\cite{HKOTY99, HKOTT02} holds:
\[
X(B, \lambda; q) = \sum_{b \in \p(B, \lambda)} q^{D(b)} = \sum_{(\nu, J) \in \rc(L, \lambda)} q^{\cc(\nu, J)} = M(L, \lambda; q).
\]
\end{corollary}

\begin{theorem}
The map $\Phi \colon B \to \RC(L)$ is a $U_q(\g_0)$-crystal isomorphism, where $\RC(L)$ is the closure of $\rc(L)$ 
under the crystal operators of~\cite{Sch:2006, schilling.scrimshaw.2015}.
\end{theorem}

\begin{proof}
It is known that the claim holds in type $A_n^{(1)}$~\cite{DS06, KSS:2002} and in type $D_n^{(1)}$~\cite{OSSS16, Sak:2013}.
By Theorem~\ref{theorem.main} and~\cite[Prop.~6.4]{schilling.scrimshaw.2015}, the map $\Phi$ commutes with the 
crystal operators. Therefore, the claim follows from Theorem~\ref{theorem.main}.
\end{proof}

%%%%%%%%%%%%%%%%%%%%%%%%%%%%%%%%%%%%%%%%%%%%%%%%%%%
\bibliographystyle{alpha}
\bibliography{paper_all}{}
 
\end{document}